\documentclass[11pt]{amsart} \textwidth=14.5cm \oddsidemargin=1cm
\usepackage{amsmath,wasysym}
\usepackage{amsxtra}
\usepackage{amscd}
\usepackage{amsthm}
\usepackage{amsfonts}
\usepackage{amssymb}
\usepackage{arydshln}
\usepackage{eucal}
\usepackage{graphics, color}
\usepackage{hyperref,mathrsfs}
\usepackage[usenames,dvipsnames]{xcolor}
\usepackage{tikz}
\usepackage{stmaryrd}
\usepackage[all]{xy}
\usepackage{hyperref}
\usepackage{color}
\usepackage{bbm} 
\allowdisplaybreaks

\hypersetup{colorlinks,linkcolor=blue, urlcolor=blue,
citecolor=blue}

\newtheorem{thm}{Theorem} [section]
\newtheorem{lem}[thm]{Lemma}
\newtheorem{cor}[thm]{Corollary}
\newtheorem{prop}[thm]{Proposition}

\newtheorem{rmk}[thm]{Remark}

\theoremstyle{definition}

\numberwithin{equation}{section}
\newcommand{\A}{\mathcal A}

\newcommand{\alg}{\textup{alg}}
\newcommand{\ba}[1]{    \begin{array}   #1   \end{array}}
\newcommand{\bc}[1]{     \begin{cases}   #1   \end{cases}}
\newcommand{\Bp}[1]{\Big{(} #1\Big{)}}

\newcommand{\co}{\mathrm{col}}

\newcommand{\D}{\mathcal{D}}

\newcommand{\ds}{\displaystyle}

\newcommand{\End}{{\mathrm{End}}}
\newcommand{\fa}{\mathfrak{a}}
\newcommand{\fc}{\mathfrak{c}}

\newcommand{\Hom}{{\mathrm{Hom}}}
\newcommand{\id}{\mathbbm{1}}

\newcommand{\ld}{\lambda}
\newcommand{\Ld}{\Lambda}
\newcommand{\lrb}[2]{{#1 \brack #2}}
\newcommand{\NN}{\mathbb{N}}

\newcommand{\ro}{\mathrm{row}}

\newcommand{\ZZ}{\mathbb{Z}}
\renewcommand{\^}[1]{\widehat{#1}}

\renewcommand{\=}[1]{\overline{#1}}

\newcommand{\nc}{\newcommand}
\nc{\browntext}[1]{\textcolor{brown}{#1}}
\nc{\greentext}[1]{\textcolor{green}{#1}}
\nc{\redtext}[1]{\textcolor{red}{#1}}
\nc{\bluetext}[1]{\textcolor{blue}{#1}}
\nc{\brown}[1]{\browntext{ #1}}
\nc{\green}[1]{\greentext{ #1}}
\nc{\red}[1]{\redtext{ #1}}
\nc{\blue}[1]{\bluetext{ #1}}

\title[$\imath$Schur superalgebras]
{Multiplication formulas and isomorphism theorem of $\imath$Schur superalgebras}

\author[Jian Chen]{Jian Chen}
\address{School of mathematical Sciences,
Shanghai Key Laboratory of Pure Mathematics and Mathematical Practice,
    East China Normal University, Shanghai 200241, China}
    \email{52215500009@stu.ecnu.edu.cn (Chen)}

\author[Li Luo]{Li Luo}
\address{School of mathematical Sciences,
Shanghai Key Laboratory of Pure Mathematics and Mathematical Practice,
    East China Normal University, Shanghai 200241, China}
\email{lluo@math.ecnu.edu.cn (Luo)}


\begin{document}

\begin{abstract}
We introduce the notion of $\imath$Schur superalgebra, which can be regarded as a type B/C counterpart of the $q$-Schur superalgebra (of type A) formulated as centralizer algebras of certain signed $q$-permutation modules over Hecke algebras. Some multiplication formulas for $\imath$Schur superalgebra are obtained to construct their canonical bases. Furthermore, we established an isomorphism theorem between the $\imath$Scuhr superalgebras and the $q$-Schur superalgebras of type A, which helps us derive semisimplicity criteria of the $\imath$Schur superalgebras.
\end{abstract}

\maketitle
\setcounter{tocdepth}{1}
\tableofcontents

\section{Introduction}
\subsection{}
The original Schur algebras arise in the celebrated Schur $(GL(n),\mathfrak{S}_d)$-duality between the general linear group $GL(n)$ and the symmetric group $\mathfrak{S}_d$. Their quantum analogues, named $q$-Schur algebras or quantum Schur algebras, were firstly introduced in terms of permutation modules of the Hecke algebras (of type A) by Dipper-James \cite{DJ89}, although the quantum Schur duality between the quantum general linear group $U_q(\mathfrak{gl}_n)$ and the Hecke algebra $\mathcal{H}(\mathfrak{S}_d)$ was established by Jimbo \cite{Jim86} much earlier. Beilinson-Lusztig-MacPherson \cite{BLM90} provided a geometric realization of $q$-Schur algebras via partial flag varieties. Furthermore, they derived some multiplication formulas and established a stabilization property, by which the (modified) quantum general linear group $\dot{U}_q(\mathfrak{gl}_n)$ and its canonical basis are constructed.

Nowadays there have been various generalizations of Schur algebras and $q$-Schur algebras in the literature. Donkin \cite{D86,D87} defined and studied a family of generalized Schur algebras by replacing $GL(n)$ with a reductive group $G$ of arbitrary type, whose $q$-deformations were formulated by Doty \cite{Do03}. A version of $q$-Schur algebras as centralizer algebras of certain modules over Hecke algebras of type B was formulated independently by Dipper-James-Mathas \cite{DJM98} and Du-Scott \cite{DS00} (see also \cite{DJM98b} for a further cyclotomic generalization). Another different generalization of the $q$-Schur algebras related to Hecke algebras of type B was introduced earlier by Green \cite{Gr97}. Green's $q$-Schur algebras are one of the two classes of $\imath$Schur algebras used by Bao-Wang in \cite{BW18} related to $\imath$quantum groups. These $\imath$Schur algebras also admit a geometric realization \`a la BLM in terms of partial flags of type B/C \cite{BKLW18} (see \cite{FLLLWa,FLLLWb} for affine type C). Such geometric approach further provides a realization of the (modified) $\imath$quantum groups of types AIII/AIV without black nodes (in the sense of Satake diagrams) and their canonical bases. More developments on $\imath$Schur algebras (e.g. cellularity, quasi-hereditariness, semisimplicity and representation type) were derived by Lai-Nakano-Xiang in \cite{LNX20}.
Actually, Lai-Nakano-Xiang showed that their (2-parameters) quantum Schur algebra $S_{Q,q}^B(n, d)$ yields a concrete realization of the quasi-hereditary one-cover for $\mathcal{H}_{Q,q}^B(d)$ (as defined by Rouquier \cite{R08}). This gives compelling evidence that this Schur algebra is indeed the most natural candidate for a Schur algebra of type B. In particular, under favorable conditions, they proved a Morita equivalence between the representation theory of $S_{Q,q}^B(n, d)$ and the category ${\mathcal O}$ for rational Cherednik algebras, and introduced a Schur-type functor that identifies with the type B Knizhnik-Zamolodchikov functor. A BLM type realization for these 2-parameters Schur algebras and their related $\imath$quantum groups was achieved in \cite{LL21}.
We refer to \cite{Bao17,FL15} for the type D analogue.
In fact, all $q$-Schur algebras studied in \cite{DJ89,Gr97,BKLW18,LNX20,LL21} are special cases of the $\imath$Schur algebras formulated by Shen-Wang \cite{SW21}. There is also a generalization of $q$-Schur algebras which is valid for the Hecke algebras of all finite types in \cite{LW22} (see also \cite{CLW20} for all affine types).

Schur superalgebras were firstly introduced in \cite{Mu91} and their representation theory were developed in \cite{D01,BK03}. Their quantum analogues (called $q$-Schur superalgebras) were defined by Mitsuhashi \cite{Mit06} via the $q$-tensor superspace, which coincide with the ones derived by Du-Rui \cite{DR11} via certain signed $q$-permutation modules for the Hecke algebras (of type A). Moreover, we refer to \cite{EK13} for a presentation of the $q$-Schur superalgebras, \cite{DGW14,DGW17} for the classification of their irreducible modules and \cite{DGZ20} for semisimplicity criteria.

\subsection{}
Our paper aims to introduce and investigate the $\imath$Schur superalgebras, the super counterpart of $\imath$Schur algebras. They can also be regarded as the type B analogue of the $q$-Schur algebras formulated in \cite{DR11} since they are defined as centralizer algebras of certain signed $q$-permutation modules for the Hecke algebras of type B. We obtain some closed multiplication formulas for simple generators, by which we are able to construct monomial bases and hence canonical bases of $\imath$Schur superalgebras.
 We also establish an isomorphism theorem between the $\imath$Schur superalgebras and the $q$-Schur superalgebras of type A (under an invertible condition), so that we can give semisimplicity criteria for $\imath$Schur superalgebras via a similar argument to the one proposed by Lai-Nakano-Xiang in \cite{LNX20}.

 The paper is organized as follows. In Section~2, we introduce the notion of $\imath$Schur superalgebras $\mathcal{S}^\jmath_{m|n,d}$, which are defined as centralizer algebras of certain signed $q$-permutation modules for the Hecke algebras of type $\mathrm{B}_d$. We prove that the direct sum of these signed $q$-permutation modules is isomorphic to the $q$-tensor superspace as modules of Hecke algebras. Section~3 is devoted to obtaining some closed multiplication formulas for simple generators of $\mathcal{S}^\jmath_{m|n,d}$. These formulas help us construct a monomial basis and hence a canonical basis for $\mathcal{S}^\jmath_{m|n,d}$ in Section~4. In Section~5, we show an isomorphism theorem between $\imath$Schur superalgebras and $q$-Schur superalgebras of type A, which is a super generalization of the isomorphism theorem between $\imath$Schur algebras and $q$-Schur algebras of type A due to Lai-Nakano-Xiang \cite{LNX20}. In the final section, we introduce a variant of $\imath$Schur superalgebras $\mathcal{S}^\imath_{m|n,d}$ and provide their canonical bases and semisimplicity criteria, too.

\subsection{}
Finally, let us discuss below several interesting questions arising from this work.

As we mentioned before, multiplication formulas for ($\imath$-)Schur algebras can be employed to gain a stabilization property in the spirit of the work \cite{BLM90} to realize ($\imath$-)quantum groups and their canonical bases. It will be interesting to see if one can introduce super analogues of $\imath$quantum groups, which can be realized via our $\imath$Schur superalgebras and the closed multiplication formulas obtained in this paper. The theory of canonical bases for these $\imath$quantum supergroups can also be established.
Furthermore, we may ask whether there exist other $\imath$quantum supergroups. It leads us to a natural problem: to classify super quantum symmetric pairs.

In another direction, though it has been showed in \cite{BW18} that $\imath$quantum groups suffice to formulate Kazhdan-Lusztig theory for ortho-symplectic Lie superalgebras, we hope our work can help to develop another direct connection with representation theory of Lie superalgebras.

\subsection*{Notational convention} Throughout the paper, we fix $d\in\mathbb{N}\setminus\{0\}$ and $m,n\in\mathbb{N}$. Denote $$D=2d+1 \quad \mbox{and}\quad N=2m+2n+1.$$ For any $a,b\in\mathbb{Z}$ with $a\leq b$, we always denote
$$[a..b]:=[a,b]\cap\mathbb{Z}=\{a,a+1,\ldots,b\}.$$

\subsection*{Acknowledgement}
LL is partially supported by the NSF of China (grant No. 11871214) and the Science and Technology Commission of Shanghai Municipality (grant No. 18dz2271000, 21ZR1420000).


\section{Hecke algebras and $\imath$Schur superalgebras}

\subsection{Weyl groups}

Let $W$ be the Weyl group of type $\mathrm{B}_d$ generated by $S=\{s_0,s_1,\ldots,s_{d-1}\}$. It can be identified as a permutation group of $[-d..d]$, i.e.
$$W=\{w\in\mathrm{Perm}[-d..d]~|~w(-i)=-w(i)\}$$
with
$s_0=(-1,1), s_i=(i,i+1)(-i,-i-1), (1\leq i<d)$.
It is obvious that $w(0)=0$ for any $w\in W$.
We denote the identity of $W$ by $\id$.

The length function $\ell:W\rightarrow \mathbb{N}$ can be described by the following lemma.
\begin{lem}\label{lem:length}
  For any $w\in W$, we have
  $$\ell(w)=\frac{1}{2}\sharp\left\{(i,j)\in[1..d]\times[-d..d]~\middle|~\substack{i< j\\ w(i) > w(j)} \textup{ or }\substack{i > j\\ w(i) < w(j)}\right\}.$$
\end{lem}
\begin{proof}
  See \cite[Lemma~2.1.1]{LL21}.
\end{proof}

\subsection{Compositions with parity}
Denote
\begin{align*}
\mathbb{I}:=\mathbb{I}_{0}\cup\mathbb{I}_{1}=[-n-m..m+n],
\end{align*}
where $\mathbb{I}_{0}:=[-m..m]$ and $\mathbb{I}_{1}:=[-n-m..-1-m]\sqcup[m+1..m+n]$.
The parity of a number $i\in\mathbb{I}$ is defined by
$\widehat{i} = \bc{0, &\textup{if}\,\,i\in \mathbb{I}_0;\\
1, &\textup{if}\,\,i\in\mathbb{I}_1.}$

We denote the set of weak compositions (with parity) of $d$ into $m+n+1$ parts by
\begin{align*}
&\Lambda(m|n,d):=\\
&\big\{\lambda=(\lambda^{(0)}|\lambda^{(1)})=(\lambda_0,\lambda_1,\ldots,\lambda_m|\lambda_{m+1},\ldots,\lambda_{m+n})\in\mathbb{N}^{m+n+1}~\big|~\sum_{i=0}^{m+n}\lambda_i=d\big\}.
\end{align*}
For any $\lambda\in\Lambda(m|n,d)$ and $i\in\mathbb{I}$, define the integral intervals
\begin{equation*}
R^\lambda_i  = \bc{
\{-\ld_{0},\cdots,-1,0,1,\cdots,\ld_{0}\}&\textup{if}\quad i=0;
\\
\{\widetilde{\ld}_{i-1}+1,\widetilde{\ld}_{i-1}+2,\cdots,\widetilde{\ld}_{i}\}&\textup{if}\quad 1\leq i\leq m+n;
\\
-R^\ld_{-i}&\textup{if}\quad -m-n\leq i\leq-1,
}\qquad \mbox{where $\widetilde{\ld}_i = \sum\limits_{j=0}^i\ld_j$}.
\end{equation*}

\subsection{Parabolic subgroups with parity}
For $\lambda\in\Lambda(m|n,d)$, denote by $W_\lambda$ the parabolic subgroup of $W$ generated by $S\setminus\{s_{\widetilde{\lambda}_0},s_{\widetilde{\lambda}_1},\ldots,s_{\widetilde{\lambda}_{m+n-1}}\}$. It is clear that $w\in W_\ld$ if and only if $w(R_i^\ld)=R_i^\ld$ for any $i\in\mathbb{I}$. Moreover,
we have
\begin{equation}\label{eq:Wla}
  W_\lambda=W_{\lambda^{(0)}}W_{\lambda^{(1)}}\cong W_{\lambda^{(0)}}\times W_{\lambda^{(1)}},
\end{equation}
where $W_{\lambda^{(0)}}\leq \langle s_0,s_1,\ldots,s_{\widetilde{\ld}_m}\rangle$ and $W_{\lambda^{(1)}}\leq \langle s_{\widetilde{\ld}_m+1},\ldots,s_{d-1}\rangle$ are called the \emph{even} and \emph{odd} parts of $W_\ld$, respectively. So for each $w\in W_\lambda$, we can uniquely write $w=w^{(0)}w^{(1)}$ with $w^{(0)}\in W_{\ld^{(0)}}, w^{(1)}\in W_{\ld^{(1)}}$.

Let $\mathcal{D}_\ld$ be the set of shortest right coset representatives for $W_\ld\setminus W$. It is known (cf. \cite{LL21}) that
\begin{equation}\label{dld}
   \mbox{$g\in\mathcal{D}_\ld$ if and only if $g^{-1}$ keeps order on each integral interval $R_i^\ld$}.
 \end{equation}
 Let $\mathcal{D}_{\ld\mu}$ (resp. $\mathcal{D}_{\ld^{(0)}\mu^{(0)}}$ and $\mathcal{D}_{\ld^{(1)}\mu^{(1)}}$) be the set of shortest double coset representatives of $W_\ld\setminus W/W_\mu$ (resp. $W_{\ld^{(0)}}\setminus W/W_{\mu^{(0)}}$ and $W_{\ld^{(1)}}\setminus W/W_{\mu^{(1)}}$). It is clear that
\begin{equation}\label{Dsubset}
\mathcal{D}_{\ld\mu}\subset \mathcal{D}_{\ld^{(0)}\mu^{(0)}} \quad\mbox{and}\quad \mathcal{D}_{\ld\mu}\subset\mathcal{D}_{\ld^{(1)}\mu^{(1)}}.
\end{equation}

Set $$\Xi_{m+n,d}:=\{(a_{ij})_{i,j\in\mathbb{I}}\in\mathrm{Mat}_{\mathbb{I}\times \mathbb{I}}(\mathbb{N})~|~\sum_{i,j\in\mathbb{I}} a_{ij}=D, a_{ij}=a_{-i,-j}\}.$$
The following statement has been proved in \cite[Lemma~2.2.1]{LL21}.
\begin{lem}\label{def:kappa1}
  The map
  $$\kappa: \bigsqcup_{\ld,\mu\in\Lambda(m|n,d)} \{\ld\} \times \D_{\ld\mu} \times \{\mu\} \to \Xi_{m+n,d},
\quad
\kappa(\ld,g,\mu) = (|R_i^\ld \cap g R_j^\mu|)_{i,j\in\mathbb{I}}$$ is a bijection.
\end{lem}

 For $A=(a_{ij})_{i,j\in\mathbb{I}}\in \Xi_{m+n,d}$, denote
 \begin{align*}
 \mathrm{row}(A)&=(\frac{a_{00}-1}{2}+\sum_{j=1}^{m+n}a_{0j},\sum_{j\in\mathbb{I}}a_{1j},\ldots,\sum_{j\in\mathbb{I}}a_{m+n,j}),\\
 \mathrm{col}(A)&=(\frac{a_{00}-1}{2}+\sum_{i=1}^{m+n}a_{i0},\sum_{i\in\mathbb{I}}a_{i1},\ldots,\sum_{i\in\mathbb{I}}a_{i,m+n}).
 \end{align*}
It is not difficult to verify that $\kappa^{-1}(A)=(\mathrm{row}(A),g,\mathrm{col}(A))$. Here the element $g$ can also be read by a standard algorithm (refer to \cite[\S 2.2]{LL21}).

For $\ld,\mu\in\Lambda(m|n,d)$ and $g\in\D_{\ld\mu}$, write $A=\kappa(\ld,g,\mu)=(a_{ij})_{i,j\in\mathbb{I}}$. We define a weak composition of $d$ as follows:
write $\mathfrak{r}=2(m+n+1)(m+n)+1$,
\begin{align}\label{def:delta}
\delta&=\delta(A)=\delta(\ld,g,\mu)=(\delta_0,\delta_1,\cdots,\delta_{\mathfrak{r}})\\
&:=(\frac{a_{00}-1}{2},a_{10},\ldots,a_{m+n,0},a_{-m-n,1},a_{-m-n+1,1},\ldots,a_{m+n,1},\ldots,\nonumber\\
&\qquad\qquad\qquad\qquad\qquad
\ldots,a_{-m-n,m+n},a_{-m-n+1,m+n},\ldots,a_{m+n,m+n}).\nonumber
\end{align}
We also write $\delta=(\delta^{(0)}|\delta^{(1)})$ with
$$\delta^{(0)}=(\frac{a_{00}-1}{2},a_{10},\ldots,a_{m+n,m})\quad \mbox{and}\quad \delta^{(1)}=(a_{-m-n,m+1},\ldots,a_{m+n,m+n}).$$
A direct computation shows $$W_\delta=g^{-1}W_\lambda g\cap W_\mu=W^{00}_{\delta}\times W^{01}_{\delta}\times W^{10}_{\delta}\times W^{11}_{\delta}\leq W_\mu,$$
where $W^{ij}_{\delta}=g^{-1}W_{\ld^{(i)}}g\cap W_{\mu^{(j)}}$ for any $i,j=0,1$.
Moreover, this $\delta$ satisfies the following lemma, whose proof can be found in \cite[Lemma~4.17 \& Theorem~4.18]{DDPW08}.
\begin{lem}\label{lem:doublecoset}
\begin{enumerate}
\item The map $$\varphi: W_\ld \times (\D_\delta \cap W_\mu) \rightarrow W_\ld g W_\mu,\quad (x,y)\mapsto xgy$$ is bijective. In addition, $\ell(xgy) = \ell(x) + \ell(g) + \ell(y)$.
\item The map $$\psi: W_\delta \times (\D_\delta\cap W_\mu) \rightarrow W_\mu,\quad (x,y)\mapsto xy$$ is bijective. In addition, $\ell(x) + \ell(y) = \ell(xy)$.
\end{enumerate}
\end{lem}

\subsection{Even-odd trivial intersection property}
We say $g\in\D_{\ld\mu}$ admits the {\em even-odd trivial intersection property} if $$g^{-1}W_{\ld^{(1)}}g\cap W_{\mu^{(0)}} = g^{-1}W_{\ld^{(0)}}g\cap W_{\mu^{(1)}}=\{\id\}.$$
Let $$\D_{\ld\mu}^{\circ}:=\{g\in\D_{\ld\mu}~|~\mbox{$g$ admits the even-odd trivial intersection property}\}.$$ We shall see in the next section that $\D_{\ld\mu}^{\circ}$ is the super version of $\D_{\ld\mu}$.
If $g\in\D_{\ld\mu}^{\circ}$ then the composition $\delta=\delta(\ld,g,\mu)$, which is defined in \eqref{def:delta}, satisfies $W^{01}_{\delta}=W^{10}_{\delta}=\{\id\}$ and hence
\begin{equation}\label{eq:deltai}
W_{\delta^{(0)}}=W^{00}_{\delta}=g^{-1}W_{\ld^{(0)}}g\cap W_{\mu^{(0)}} \quad\mbox{and}\quad W_{\delta^{(1)}}=W^{11}_{\delta}=g^{-1}W_{\ld^{(1)}}g\cap W_{\mu^{(1)}}.
\end{equation}
Therefore we obtain the following lemma which means that under the even-odd trivial intersection assumption, the parity of $W_\delta$ is compatible with the parity of $W_\mu$.
\begin{lem}\label{lem:comp}
  Let $\ld,\mu\in\Lambda(m|n,d)$, $g\in\D_{\ld\mu}^{\circ}$ and $\delta=\delta(\ld,g,\mu)$.
  Each $w=w^{(0)}w^{(1)}\in W_\delta\leq W_{\mu}$ with $w^{(i)}\in W_{\mu^{(i)}}$, $(i=0,1)$, satisfies
 $w^{(i)}\in W_{\delta^{(i)}}$, $(i=0,1)$.
\end{lem}

Let
\begin{align*}
\Xi_{m|n,d}:=\{A=(a_{ij})_{i,j\in\mathbb{I}}\in\Xi_{m+n,d}~|~\mbox{$a_{ij}=0$ or $1$ if $\widehat{i}\neq\widehat{j}$}\}.
\end{align*}
Below is an analogue of Lemma~\ref{def:kappa1}.
\begin{lem}\label{lem:kappacirc}
  The map
  $$\kappa^\circ: \bigsqcup_{\ld,\mu\in\Lambda(m|n,d)} \{\ld\} \times \D_{\ld\mu}^\circ \times \{\mu\} \to \Xi_{m|n,d},
\quad
\kappa^\circ(\ld,g,\mu) = (|R_i^\ld \cap g R_j^\mu|)_{i,j\in\mathbb{I}}$$ is a bijection.
\end{lem}
\begin{proof}
For any $\ld,\mu\in\Lambda(m|n,d)$ and $g\in\D_{\ld\mu}^{\circ}$, the even-odd trivial intersection property implies $|R_i^\ld \cap g R_j^\mu|\leq 1$ if $\widehat{i}\neq\widehat{j}$. So the map $\kappa^\circ$ is well defined.

The injectivity is clear since $\kappa^\circ$ is a restriction of the bijection $\kappa$ defined in Lemma~\ref{def:kappa1}.

For any $A\in \Xi_{m|n,d}$, see \cite[\S2.2]{LL21} for the explicit construction of $\kappa^{-1}(A)$, which lies in $\bigsqcup_{\ld,\mu\in\Lambda(m|n,d)} \{\ld\} \times \D_{\ld\mu}^\circ \times \{\mu\}$. Hence $\kappa^\circ$ is surjective.
\end{proof}

We shall still denote the map $\kappa^\circ$ by $\kappa$ hereinafter.

\subsection{Hecke algebras}\label{subsection:Hecke algebras}
Let $R$ be a commutative ring with $1$ and let $q\in R$ satisfy $q\neq0,1$. The Hecke algebra $\mathcal{H}_R=\mathcal{H}_{R}(W)$ is a free $R$-module with a basis $\{T_w~|~w\in W\}$ and the multiplication defined by
\begin{align}
\label{def:ww'}
&T_w T_{w'} = T_{ww'},\quad\mbox{if $\ell(w w') = \ell(w) + \ell(w')$},
\\
\label{def:si}
&(T_{s_i}+1)(T_{s_i}-q)=0,\quad \forall\ 0\leq i<d.
\end{align}
Particularly, if we take $$R=\mathcal{A}:=\mathbb{Z}[v,v^{-1}]\quad \mbox{and} \quad q=v^2,$$ the Hecke algebra is simply denoted by $\mathcal{H}$.

Recall $W_\ld=W_{\ld^{(0)}}W_{\ld^{(1)}}$ for $\lambda=(\ld^{(0)}|\ld^{(1)})\in\Lambda(m|n,d)$ in \eqref{eq:Wla}. We let
$$\mathrm{x}_{\ld^{(0)}}=\sum_{w\in W_{\ld^{(0)}}} T_{w}, \qquad
\mathrm{y}_{\ld^{(1)}}=\sum_{w\in W_{\ld^{(1)}}} (-q)^{-\ell(w)}T_{w}\qquad\mbox{and}$$
$$[\mathrm{xy}]_{\ld} := \mathrm{x}_{\ld^{(0)}}\mathrm{y}_{\ld^{(1)}} =\sum_{w^{(0)}w^{(1)}\in W_{\ld}}T_{w^{(0)}}((-q)^{-\ell(w^{(1)})}T_{w^{(1)}}).$$

\begin{lem}\label{lem:wx}
For $w=w^{(0)}w^{(1)}\in W_{\ld}$, we have
$$T_w \mathrm{[xy]}_{\ld} = (-1)^{\ell(w^{(1)})}q^{\ell(w^{(0)})}\mathrm{[xy]}_{\ld} = \mathrm{[xy]}_{\ld}T_w.$$
\end{lem}
\begin{proof}
Firstly, we consider the case when $w=s\in S$ is a simple reflection.

If $s\in W_{\ld^{(0)}}$, then
\begin{align*}
T_{s}\mathrm{x}_{\ld^{(0)}}&=\sum_{w\in W_{\ld^{(0)}}} T_{s}T_{w}=\sum_{\tiny\substack{w\in W_{\ld^{(0)}},\\sw<w}}T_{s}(T_{w}+T_{sw})\\
&= \sum_{\tiny\substack{w\in W_{\ld^{(0)}},\\sw<w}}T_{s}(T_{s}+1)T_{sw}=\sum_{\tiny\substack{w\in W_{\ld^{(0)}},\\sw<w}}(T^2_{s}+T_{s})T_{sw}\\
&= \sum_{\tiny\substack{w\in W_{\ld^{(0)}},\\sw<w}}q(T_{s}+1)T_{sw}=q\sum_{\tiny\substack{w\in W_{\ld^{(0)}},\\sw<w}}(T_{w}+T_{sw})= q\mathrm{x}_{\ld^{(0)}}.
\end{align*}
Hence $T_{s}[xy]_\ld=T_{s}\mathrm{x}_{\ld^{(0)}}\mathrm{y}_{\ld^{(1)}}
=q\mathrm{x}_{\ld^{(0)}}\mathrm{y}_{\ld^{(1)}}=q[xy]_\ld$.

If $s\in W_{\ld^{(1)}}$, then $T_{s}\mathrm{x}_{\ld^{(0)}}=\mathrm{x}_{\ld^{(0)}}T_{s}$ and
\begin{align*}
T_{s}\mathrm{y}_{\ld^{(1)}}&=\sum_{w\in W_{\ld^{(1)}}} (-q)^{-\ell(w)}T_{s}T_{w}=\sum_{\tiny\substack{w\in W_{\ld^{(1)}},\\sw<w}}(-q)^{-\ell(w)}T_{s}(T_{w}-qT_{sw})\\
&= \sum_{\tiny\substack{w\in W_{\ld^{(1)}},\\sw<w}}(-q)^{-\ell(w)}T_{s}(T_{s}-q)T_{sw}= \sum_{\tiny\substack{w\in W_{\ld^{(1)}},\\sw<w}}(-q)^{-\ell(w)}(T^2_{s}-qT_{s})T_{sw}\\
&= \sum_{\tiny\substack{w\in W_{\ld^{(1)}},\\sw<w}}(-q)^{-\ell(w)}(-T_{s}+q)T_{sw}=-\sum_{\tiny\substack{w\in W_{\ld^{(1)}},\\sw<w}}(-q)^{-\ell(w)}(T_{w}-qT_{sw})\\
&=-\sum_{\tiny\substack{w\in W_{\ld^{(1)}},\\sw<w}}(-q)^{-\ell(w)}(T_{w}-qT_{sw})=-\mathrm{y}_{\ld^{(1)}}.
\end{align*}
Hence $T_{s}[xy]_\ld=T_{s}\mathrm{x}_{\ld^{(0)}}\mathrm{y}_{\ld^{(1)}}=\mathrm{x}_{\ld^{(0)}}T_{s}\mathrm{y}_{\ld^{(1)}}
=-\mathrm{x}_{\ld^{(0)}}\mathrm{y}_{\ld^{(1)}}=-[xy]_\ld$.
Then by the Hecke relation \eqref{def:ww'} and induction on $\ell(w)$, we can obtain that for any $w=w^{(0)}w^{(1)}\in W_{\ld}$,
$$T_w \mathrm{[xy]}_{\ld} = (-1)^{\ell(w^{(1)})}q^{\ell(w^{(0)})}\mathrm{[xy]}_{\ld}.$$
A similar argument shows $\mathrm{[xy]}_{\ld}T_w = (-1)^{\ell(w^{(1)})}q^{\ell(w^{(0)})}\mathrm{[xy]}_{\ld}$.
\end{proof}

\subsection{$\imath$Schur superalgebras}
The $\imath$Schur superalgebra is defined as the following $R$-algebra
$$\mathcal{S}^\jmath_{m|n,d;R}
:=\End_{\mathcal{H}_R}\big(\bigoplus_{\lambda\in\Lambda(m|n,d)}[\mathrm{xy}]_\lambda\mathcal{H}_R\big)
= \bigoplus_{\ld,\mu \in \Lambda(m|n, d)} \Hom_{\mathcal{H}_R} ([\mathrm{xy}]_{\ld}\mathcal{H}_R, [\mathrm{xy}]_{\mu}\mathcal{H}_R).$$
Similar to the notation for Hecke algebras, we simply denote the $\imath$Schur superalgebra by $\mathcal{S}^\jmath_{m|n,d}$ if we take $R=\mathcal{A}=\mathbb{Z}[v,v^{-1}]$ and $q=v^2$.
The super structure (i.e. $\ZZ_2$-grading) of $\mathcal{S}^\jmath_{m|n,d}$ is defined via
\begin{align*}
  \mathcal{S}^\jmath_{{m|n,d;i}}:= \bigoplus_{\tiny\substack{\ld,\mu \in \Lambda(m|n, d)\\|\ld^{(1)}|+|\mu^{(1)}|\equiv i(\mathrm{mod}~2)}} \Hom_{\mathcal{H}} ([\mathrm{xy}]_{\ld}\mathcal{H}, [\mathrm{xy}]_{\mu}\mathcal{H}), \qquad (i=0,1),
\end{align*}
where $|\ld^{(1)}|=\ld_{m+1}+\ld_{m+2}+\cdots+\ld_{m+n}$ and similar to $|\mu^{(1)}|$. In this present paper, the multiplication on $\mathcal{S}^\jmath_{m|n,d}$ is just the composition of maps. We shall not involve supermultiplication (cf. \cite[(5.8.1)]{DR11}).

For $\mu,\nu\in\Lambda(m|n,d)$ and $g\in\D^{\circ}_{\mu\nu}$, let $A=\kappa(\mu,g,\nu)$ and $\delta=\delta(A)$. Set
\begin{equation}\label{def:T}
  T_{W_{\mu}gW_{\nu}}:= \mathrm{[xy]}_{\mu} T_g T_{\D_{\delta}\cap W_{\nu}},\qquad\mbox{where}
\end{equation}
$$
  T_{\D_{\delta}\cap W_{\nu}}:=\sum_{w=w^{(0)}w^{(1)}\in \D_\delta \cap W_\nu}(-q)^{-\ell(w^{(1)})}T_{w} \qquad \mbox{with $w^{(i)}\in W_{\nu^{(i)}}$, $i=0,1$}.$$

Define a right $\mathcal{H}$-linear map $\phi_{\mu\nu}^g\in\Hom_{\mathcal{H}} \left(\bigoplus_{\ld\in\Lambda(m|n,d)}[\mathrm{xy}]_{\ld}\mathcal{H}, \mathcal{H}\right)$ via
\begin{equation}
\label{def:phimunu}
\phi_{\mu\nu}^g(\mathrm{[xy]}_{\ld}h) = \delta_{\nu,\ld}T_{W_{\mu}gW_{\nu}}h, \qquad (\ld\in\Lambda(m|n, d), h\in\mathcal{H}).
\end{equation}
Furthermore, we denote
$$e_A=\phi_{\mu\nu}^g\quad \mbox{for}\quad A=\kappa(\mu,g,\nu),$$
which is well defined thanks to Lemma~\ref{lem:kappacirc}.

\begin{lem}\label{lem:e_A}
The set $\{e_A~|~A\in\Xi_{m|n,d}\}$ forms an $\mathcal{A}$-basis of $\mathcal{S}^\jmath_{m|n,d}$. In particular,
$\mathrm{rank}(\mathcal{S}^\jmath_{m|n,d})= \sum_{k=0}^{d}
{2m^2+2n^2+2m+k\choose k}{4mn+2n\choose d-k}$.
\end{lem}
\begin{proof}
As $\mathcal{A}$-modules, $\Hom_{\mathcal{H}} ([\mathrm{xy}]_{\ld}\mathcal{H}, [\mathrm{xy}]_{\mu}\mathcal{H})\cong [\mathrm{xy}]_{\mu}\mathcal{H}\cap\mathcal{H}[\mathrm{xy}]_{\ld}$.
Imitating the proof of \cite[Proposition~5.5]{DR11}, we can show that $[\mathrm{xy}]_{\mu}\mathcal{H}\cap\mathcal{H}[\mathrm{xy}]_{\ld}$ is the free $\mathcal{A}$-submodule of $\mathcal{H}$ spanned by $\{T_{W_{\ld}gW_{\mu}}\}_{g\in\D_{\ld\mu}^{\circ}}$. Thus the first statement in the lemma is valid. Then we know that the rank of $\mathcal{S}^\jmath_{m|n,d}$ equals to the cardinality of $\Xi_{m|n,d}$, which can be computed directly.
\end{proof}


\subsection{Weight function on $\mathbb{I}^d$}
For each $\ld=(\ld_0,\ld_1,\cdots,\ld_{m+n})\in\Ld(m|n,d)$, denote
\begin{align*}
\mathbf{i}_{\ld}=\{\underbrace{0,\cdots,0}_{\ld_0},\underbrace{1,\cdots,1}_{\ld_1},\cdots,\underbrace{m+n,\cdots,m+n}_{\ld_{m+n}}\}\in\mathbb{I}^d.
\end{align*}
\begin{lem}
For any $\mathbf{i}=(i_{1},i_{2},\cdots,i_{d})\in\mathbb{I}^d$, there is a unique $\lambda\in\Ld(m|n,d)$ and $g\in\mathcal{D}_\ld$ such that $\mathbf{i}=\mathbf{i}_{\ld}g$. In fact,
$\lambda=(\ld_0,\ld_1,\ldots,\ld_{m+n})$ with
$$\ld_k=\sharp\{j~|~i_j=|k|,1\leq j\leq d\}, \quad (k=0,1,\ldots,m+n).$$
\end{lem}
\begin{proof}
  Obvious.
\end{proof}
We denote $\mathrm{wt}(\mathbf{i})=\lambda$ if $\mathbf{i}=\mathbf{i}_\lambda g$, which is called the \emph{weight} of $\mathbf{i}$.

\subsection{Parity on $\mathbb{I}^d$}
For $\mathbf{i}=(i_{1},i_{2},\cdots,i_{d})\in \mathbb{I}^d$, we define its parity $\widehat{\mathbf{i}}\in\{0,1\}$ by
\begin{equation}\label{pairtyonId}
\widehat{\mathbf{i}}\equiv\bigg(\sum_{k<l,i_{k}>i_{l}}
+\sum_{k\leq l,-i_{k}>i_{l}}\bigg) \widehat{i_{k}}\widehat{i_{l}} \mod 2.
\end{equation}
There is a right $W$-action on $\mathbb{I}^d$ defined on generators: for $\mathbf{i}=(i_{1},i_{2},\cdots,i_{d})\in \mathbb{I}^d$,
\begin{eqnarray}\label{eq:actiononId}
\mathbf{i} s_{k}=\bc{(i_{1},\cdots,i_{k+1},i_{k},\cdots,i_{d}), &\textup{if}\,\,k\neq 0;\\
(-i_{1},i_{2},\cdots,i_{d}), &\textup{if}\,\,k=0.}
\end{eqnarray}

\begin{lem} \label{lem:parity}
The following equation holds for any $\mathbf{i}=(i_{1},i_{2},\cdots,i_{d})\in \mathbb{I}^d$:
  \begin{align*}
  (-1)^{\widehat{\mathbf{i}s_a}}=
  \left\{
  \begin{array}{ll}
  (-1)^{\widehat{\mathbf{i}}}(-1)^{\widehat{i_{a}}\widehat{i_{a+1}}}, & a=1,2,\ldots,m+n-1;\\
  (-1)^{\widehat{\mathbf{i}}}(-1)^{\widehat{i_{1}}}, & a=0.
  \end{array}\right.
\end{align*}
\end{lem}
\begin{proof}
\textbf{Case} 1: $a=1,2,\ldots,m+n-1$.
\\
Let $\Sigma_1$ and $\Sigma_2$ be the sums $\sum_{k<l,i_{k}>i_{l}}$ and $\sum_{k\leq l,-i_{k}>i_{l}}$ in \eqref{pairtyonId} corresponding to $\widehat{\mathbf{i}}$, respectively. Similarly, let
$\Sigma_1'$ and $\Sigma_2'$ be the same sums corresponding to $\widehat{\mathbf{i}s_a}$. Then $\Sigma_2=\Sigma_2'$ since ${i_{a}}+{i}_{a+1}={i}_{a+1}+{i}_{a}$. Moreover, we have
either $\Sigma_1+\widehat{i_{a}}\widehat{i_{a+1}}=\Sigma_1'$ if ${i_{a}}<{i}_{a+1}$ or $\Sigma_1=\Sigma_1'+\widehat{i_{a}}\widehat{i_{a+1}}$ if ${i_{a}}>{i}_{a+1}$.

\textbf{Case 2:} $a=0$.
\\
The formula is trivial if $i_{1}=0$. So we assume $i_{1}>0$ (the case of $i_{1}<0$ is similar). Let
$\lozenge=\bigg(\sum\limits_{k=2}^{d}\sum\limits_{k<l,i_{k}>i_{l}}
+\sum\limits_{k=2}^{d}\sum\limits_{k\leq l,-i_{k}>i_{l}}\bigg) \widehat{i_{k}}\widehat{i_{l}}$, then
\begin{align*}
\widehat{\mathbf{i}}=\lozenge+(\widehat{i_{1}}\widehat{i_{k_1}}+\widehat{i_{1}}\widehat{i_{k_2}}+\dots+\widehat{i_{1}}\widehat{i_{k_s}})
+(\widehat{i_{1}}\widehat{i_{l_1}}+\widehat{i_{1}}\widehat{i_{l_2}}+\cdots+\widehat{i_{1}}\widehat{i_{l_t}}),
\end{align*}
where $\widehat{i_{1}}\widehat{i_{k_1}}+\widehat{i_{1}}\widehat{i_{k_2}}+\cdots+\widehat{i_{1}}\widehat{i_{k_s}}=\sum_{k_j>1,i_{1}>i_{k_j}}\widehat{i}_{1}\widehat{i_{k_j}}$ and $\widehat{i_{1}}\widehat{i_{l_1}}+\widehat{i_{1}}\widehat{i_{l_2}}+\cdots+\widehat{i_{1}}\widehat{i_{l_t}}=\sum_{l_j\geq1, -i_{1}>i_{l_j}}\widehat{i_{1}}\widehat{i_{l_j}}$. Without loss of generality, we suppose $i_{1}> i_{k_h},i_{k_{h+1}},\ldots,i_{k_s} \geq -i_{1}$, then
\begin{align*}
\widehat{\mathbf{i}s_0}=&\lozenge+(\widehat{-i_{1}}\widehat{i_{k_1}}+\widehat{-i_{1}}\widehat{i_{k_2}}+\dots+\widehat{-i_{1}}\widehat{i_{k_{h-1}}})
+\bigg((\widehat{-i_{1}}\widehat{i_{l_1}}+\widehat{-i_{1}}\widehat{i_{l_2}}+\dots+\widehat{-i_{1}}\widehat{i_{l_t}})+\\
&+(\widehat{-i_{1}}\widehat{-i_{1}}+\widehat{-i_{1}}\widehat{i_{k_h}}+\widehat{-i_{1}}\widehat{i_{k_{h+1}}}+\dots+\widehat{-i_{1}}\widehat{i_{k_s}})\bigg).
\end{align*}
Therefore $\widehat{\mathbf{i}s_0}=\widehat{\mathbf{i}}+\widehat{-i_{1}}\widehat{-i_{1}}=\widehat{\mathbf{i}}+\widehat{i_{1}}$. We have done.
\end{proof}

\subsection{The $q$-tensor superspace}
Let $V_{m|n}$ be a free $\mathcal{A}$-module with a basis $\{e_{i}~|~i\in\mathbb{I}\}$. 
The parity function $\widehat{~}$ on $\mathbb{I}$ naturally gives a $\ZZ_2$-grading on $V_{m|n}=V_{m|n}^0\oplus V_{m|n}^1$ where $V_{m|n}^0$ and $V_{m|n}^1$ are spanned by
$\{e_{i}~|~i\in\mathbb{I}_{0}\}$ and $\{e_{i}~|~i\in\mathbb{I}_{1}\}$, respectively.

For any $a\in\mathbb{Z}_+$ and $\mathbf{i}=(i_{1},i_{2},\cdots,i_{a})\in \mathbb{I}^a$, define its length by
$$|\mathbf{i}|=\sharp\{(k,l)\in[1..m+n]^2~|~k<l,i_{k}>i_{l}\}+\sharp\{(k,l)\in[1..m+n]^2~|~k\leq l, -i_{k}>i_{l}\}.$$
We denote
\begin{align}\label{def:mathbf{e}_i}e_{\mathbf{i}}=e_{i_{1}}\otimes e_{i_{2}}\otimes \cdots \otimes e_{i_{a}}\quad\mbox{and}\quad
\mathbf{e}_{\mathbf{i}}=v^{|\mathbf{i}|}e_{\mathbf{i}}=v^{|\mathbf{i}|}e_{i_{1}}\otimes e_{i_{2}}\otimes\cdots \otimes e_{i_{a}}.
\end{align}
In particular,
\begin{equation}\label{eq:ei}
\mathbf{e}_i=\left\{\begin{array}{ll}
e_i, &\mbox{if $i\geq0$},\\
 ve_i, &\mbox{if $i<0$}\end{array}\right. \quad\mbox{for $i\in\mathbb{I}$}.
\end{equation}
It is clear that $\{e_{\mathbf{i}}~|~\mathbf{i}\in \mathbb{I}^d\}$ and $\{\mathbf{e}_{\mathbf{i}}~|~\mathbf{i}\in \mathbb{I}^d\}$ form two $\mathcal{A}$-bases of $V^{\otimes d}_{m|n}$.

Let $\mathbf{i}=(i_{1},i_{2},\cdots,i_{a})\in \mathbb{I}^a$ and $\mathbf{j}=(j_{1},j_{2},\cdots,j_{b})\in \mathbb{I}^b$. We define the product $\ast$ by
\begin{align}\label{mult:mathbf{e}_i}
\mathbf{e}_{\mathbf{j}}\ast\mathbf{e}_{\mathbf{k}}:=v^{|(\mathbf{j,k})|}e_{j_{1}}\otimes\cdots \otimes e_{j_{a}}\otimes e_{k_{1}}\otimes\cdots \otimes e_{k_{b}}=\mathbf{e}_{(\mathbf{j},\mathbf{k})},
\end{align} where we write $(\mathbf{j},\mathbf{k})=(j_{1},\cdots,j_{a},k_{1},\cdots,k_{b})\in \mathbb{I}^{a+b}$.

\subsection{Action on $V_{m|n}^{\otimes d}$}
Recall the right $W$-action on $\mathbb{I}^d$ in \eqref{eq:actiononId}, which helps us formulate a $\mathcal{H}$-module structure on $V_{m|n}^{\otimes d}$ as follows.

\begin{prop}\label{lem:action_H}
There is a right action of Hecke algebra $\mathcal{H}$ on $V_{m|n}^{\otimes d}$ by, for $1\leq k<d$,
\begin{equation}\label{action:T_k}
\mathbf{e}_{\mathbf{i}}\cdot T_{s_k}=\bc{(-1)^{\widehat{i_k}\widehat{i_{k+1}}}\mathbf{e}_{\mathbf{i}s_k}, &\textup{if}\,\,i_{k}<i_{k+1};\\
\frac{(-1)^{\widehat{i_k}}(q+1)+(q-1)}{2}\mathbf{e}_{\mathbf{i}}=(-1)^{\widehat{i_k}}q^{\frac{1+(-1)^{\widehat{i_k}}}{2}}\mathbf{e}_{\mathbf{i}}, &\textup{if}\,\,i_{k}=i_{k+1};\\
(-1)^{\widehat{i_k}\widehat{i_{k+1}}}q\mathbf{e}_{\mathbf{i}s_k}+(q-1)\mathbf{e}_{\mathbf{i}}, &\textup{if}\,\,i_{k}>i_{k+1}} \qquad\mbox{and}
\end{equation}
\begin{equation}\label{action:T_0}
\mathbf{e}_{\mathbf{i}}\cdot T_{s_0}=\left\{\begin{array}{ll}
(-1)^{\widehat{i_1}}\mathbf{e}_{\mathbf{i}s_0}, &\textup{if $i_{1}>0$};\\
q\mathbf{e}_{\mathbf{i}}, &\textup{if $i_{1}=0$};\\
(-1)^{\widehat{i_1}}q\mathbf{e}_{\mathbf{i}s_0}+(q-1)\mathbf{e}_{\mathbf{i}}, &\textup{if $i_{1}<0$}.
\end{array}\right.
\end{equation}
\end{prop}
\begin{proof}
It is a known result due to Mitsuhashu \cite{Mit06} that the formula \eqref{action:T_k} defines a right action of Hecke algebra of type A. So we only need to check the relations involving $T_{s_0}$. That is,
\begin{equation}\label{relT0}
T_{s_0}^{2}=(q-1)T_{s_0}+q, \quad
T_{s_0}T_{s_k}=T_{s_k}T_{s_0}, (k>1),\qquad\mbox{and}
\end{equation}
\begin{equation}\label{braid}
T_{s_0}T_{s_1}T_{s_0}T_{s_1}=T_{s_1}T_{s_0}T_{s_1}T_{s_0}.
\end{equation}

The relations in \eqref{relT0} can be checked easily. For the braid relation in \eqref{braid},
it suffices to consider the case $d=2$
since $T_{s_0}$ and $T_{s_1}$ have no effect on the positions other than the first two of the tensor product.
The argument is case by case. For example,
\begin{itemize}
  \item if $i>j>0$, then
\begin{align*}
\mathbf{e}_{(i,j)}T_{s_0}T_{s_1}T_{s_0}T_{s_1}&=(-1)^{\widehat{i}+\widehat{j}}q\mathbf{e}_{(-i,-j)}+(-1)^{\widehat{i}+\widehat{i}\widehat{j}+\widehat{j}}(q-1)\mathbf{e}_{(-j,-i)}\\
&=\mathbf{e}_{(i,j)}T_{s_1}T_{s_0}T_{s_1}T_{s_0};
\end{align*}
\item if $0>i>j$, then
\begin{align*}
\mathbf{e}_{(i,j)}T_{s_0}T_{s_1}T_{s_0}T_{s_1}&=(-1)^{\widehat{i}+\widehat{j}}q^4\mathbf{e}_{(-i,-j)}
+(-1)^{\widehat{i}}q(q-1)[q+(q-1)^2]\mathbf{e}_{(-i,j)}\\
&~~~+(-1)^{\widehat{i}+\widehat{i}\widehat{j}+\widehat{j}}q^3(q-1)\mathbf{e}_{(-j,-i)}+(-1)^{\widehat{i}\widehat{j}+\widehat{j}}q^2(q-1)^2\mathbf{e}_{(-j,i)}\\
&~~~+(-1)^{\widehat{i}\widehat{j}}q(q-1)[q+(q-1)^2]\mathbf{e}_{(j,i)}+(-1)^{\widehat{i}+\widehat{i}\widehat{j}}q^2(q-1)^2\mathbf{e}_{(j,-i)}\\
&~~~+(q-1)^2(q^2-1)\mathbf{e}_{(i,j)}+(-1)^{\widehat{j}}q^3(q-1)\mathbf{e}_{(i,-j)}\\
&=\mathbf{e}_{(i,j)}T_{s_1}T_{s_0}T_{s_1}T_{s_0};
\end{align*}
\item if $i=j<0$, then
\begin{align*}
\mathbf{e}_{(i,i)}T_{s_0}T_{s_1}T_{s_0}T_{s_1}=&(-1)^{\widehat{i}}q^{(2+(-1)^{\widehat{i}})}(q-1)\mathbf{e}_{(-i,i)}
+(-1)^{\widehat{i}}q^{\frac{5+(-1)^{\widehat{i}}}{2}}(q-1)\mathbf{e}_{(i,-i)}\\
&+[(-1)^{\widehat{i}}q^{\frac{3+(-1)^{\widehat{i}}}{2}}(q-1)+q^{(1+(-1)^{\widehat{i}})}(q-1)^2]\mathbf{e}_{(i,i)}+q^{\frac{7+(-1)^{\widehat{i}}}{2}}\mathbf{e}_{(-i,-i)}\\
=&q(q-1)[(-1)^{\widehat{i}}q+q^{\frac{1+(-1)^{\widehat{i}}}{2}}(q-1)]\mathbf{e}_{(-i,i)}
+(-1)^{\widehat{i}}q^{\frac{5+(-1)^{\widehat{i}}}{2}}(q-1)\mathbf{e}_{(i,-i)}\\
&+[(-1)^{\widehat{i}}q^{\frac{3+(-1)^{\widehat{i}}}{2}}(q-1)+q^{(1+(-1)^{\widehat{i}})}(q-1)^2]\mathbf{e}_{(i,i)}+q^{\frac{7+(-1)^{\widehat{i}}}{2}}\mathbf{e}_{(-i,-i)}\\
=&\mathbf{e}_{(i,i)}T_{s_1}T_{s_0}T_{s_1}T_{s_0},
\end{align*}
where we use the equation $(-1)^{\widehat{i}}q+q^{\frac{1+(-1)^{\widehat{i}}}{2}}(q-1)=(-1)^{\widehat{i}}q^{(1+(-1)^{\widehat{i}})}$ for $\widehat{i}=0,1$.
\end{itemize}

The remaining cases are omitted since they can be checked similarly.
\end{proof}

\begin{lem}\label{eq:eieg}
For any $\mathbf{i}={\mathbf{i}_{\mathrm{wt}(\mathbf{i})}} g\in\mathbb{I}^d$, we have
$\mathbf{e}_{\mathbf{i}}=(-1)^{\widehat{\mathbf{i}}}\mathbf{e}_{\mathbf{i}_{\mathrm{wt}(\mathbf{i})}}T_g$.
\end{lem}
\begin{proof}
  When we compute $\mathbf{e}_{\mathbf{i}_{\mathrm{wt}(\mathbf{i})}}T_g$ by a reduced form $T_g=T_{s_{i_1}}\cdots T_{s_{i_a}}$, only the first case of \eqref{action:T_k} or \eqref{action:T_0} occurs repeatedly. Therefore the statement follows from the definition of $\widehat{\mathbf{i}}$ in \eqref{pairtyonId} directly.
\end{proof}

\subsection{A right $\mathcal{H}$-module isomorphism}
\begin{thm}
There is a right $\mathcal{H}$-module isomorphism:
\begin{align*}
\phi: V_{m|n}^{\otimes d}\rightarrow \bigoplus_{\lambda\in\Lambda(m|n,d)}[\mathrm{xy}]_\lambda\mathcal{H},\quad
(-1)^{\widehat{\mathbf{i}}} \mathbf{e}_{\mathbf{i}}\mapsto [\mathrm{xy}]_\mu T_{g},
\end{align*}
where $\mu=\mathrm{wt}(\mathbf{i})$ and $g\in\mathcal{D}_\mu$ with $\mathbf{i}=\mathbf{i}_\mu g$.
\end{thm}
\begin{proof}
Write $\mathbf{i}=(i_{1},i_{2},\cdots,i_{d})\in \mathbb{I}^d$.
Proposition~\ref{lem:action_H} implies that, for $0\leq k\leq d-1$,
\begin{eqnarray*}
(-1)^{\widehat{\mathbf{i}}}\mathbf{e}_{\mathbf{i}}\cdot T_{k}=\bc{
(-1)^{\widehat{\mathbf{i}}}(-1)^{\widehat{i_k}\widehat{i_{k+1}}}\mathbf{e}_{\mathbf{i}s_k}, &\textup{if}\ i_{k}<i_{k+1}, k>0;\\
q(-1)^{\widehat{\mathbf{i}}}\mathbf{e}_{\mathbf{i}}, &\textup{if}\ i_{k}=i_{k+1}\in \mathbb{I}_{0}, k>0;\\
-(-1)^{\widehat{\mathbf{i}}}\mathbf{e}_{\mathbf{i}}, &\textup{if}\ i_{k}=i_{k+1}\in\mathbb{I}_{1}, k>0;\\
q(-1)^{\widehat{\mathbf{i}}}(-1)^{\widehat{i_k}\widehat{i_{k+1}}}\mathbf{e}_{\mathbf{i}s_k}+(q-1)(-1)^{\widehat{\mathbf{i}}}\mathbf{e}_{\mathbf{i}}, &\textup{if}\ i_{k}>i_{k+1}, k>0;\\
(-1)^{\widehat{\mathbf{i}}}(-1)^{\widehat{i_1}}\mathbf{e}_{\mathbf{i}s_0}, &\textup{if}\ i_{1}>0, k=0;\\
q(-1)^{\widehat{\mathbf{i}}}\mathbf{e}_{\mathbf{i}}, &\textup{if}\ i_{1}=0, k=0;\\
q(-1)^{\widehat{\mathbf{i}}}(-1)^{\widehat{i_1}}\mathbf{e}_{\mathbf{i}s_0}+(q-1)(-1)^{\widehat{\mathbf{i}}}\mathbf{e}_{\mathbf{i}}, &\textup{if}\ i_{1}<0, k=0.
}
\end{eqnarray*}
On the other hand,
\begin{eqnarray*}
[\mathrm{xy}]_\mu T_{g} T_{k}=\bc{
[\mathrm{xy}]_\mu T_{g s_k}, &\mbox{if $g\prec gs_k\in\mathcal{D}_\mu$, $k>0$};\\
q[\mathrm{xy}]_\mu T_{g}, &\mbox{if $gs_k=s_l g$, $s_l\in W_{\mu^{(0)}}$, $k>0$};\\
-[\mathrm{xy}]_\mu T_{g}, &\mbox{if $gs_k=s_l g$, $s_l\in W_{\mu^{(1)}}$, $k>0$};\\
q[\mathrm{xy}]_\mu T_{g s_k}+(q-1)[\mathrm{xy}]_\mu T_{g}, &\mbox{if $g\succ gs_k\in\mathcal{D}_\mu$, $k>0$};\\
[\mathrm{xy}]_\mu T_{g s_0}, &\mbox{if $g\prec gs_0\in\mathcal{D}_\mu$, $k=0$};\\
q[\mathrm{xy}]_\mu T_{g}, &\mbox{if $gs_0\notin\mathcal{D}_\mu$, $k=0$};\\
q[\mathrm{xy}]_\mu T_{g s_0}+(q-1)[\mathrm{xy}]_\mu T_{g}, &\mbox{if $g\succ gs_0\in\mathcal{D}_\mu$, $k=0$}.
}
\end{eqnarray*}
Comparing the above two formulas by Lemma~\ref{lem:parity}, we see that the $\mathcal{A}$-linear map $\phi$ is a right $\mathcal{H}$-module homomorphism. Furthermore, it is clear that $\phi$ is bijective.
Hence $\phi$ is an isomorphism.
\end{proof}
{\rmk In the above computation, we use the fact that if $g\in\mathcal{D}_\mu$ but $gs_k\not\in\mathcal{D}_\mu$ then there is a unique $s_l\in W_\mu$ such that $gs_k=s_l g$. It can be proved as follows. Thanks to \eqref{dld}, the condition $g\in\mathcal{D}_\mu$ but $gs_k\not\in\mathcal{D}_\mu$ implies that $g(k)$ and $g(k+1)$ lie in the same interval $R_i^\mu$, which forces $g(k+1)=g(k)+1$. Thus taking $l=g(k)$ we have $s_l\in W_\mu$ and $gs_k=s_l g$. Particularly, if $k=0$ then $l=g(0)=0$, i.e. $gs_0=s_0 g$ in this case.}
\begin{cor}\label{cor:iso}
There is an isomorphism of superalgebras:
\begin{align*}
\mathcal{S}^\jmath_{m|n,d}=\End_{\mathcal{H}}\big(\mathop{\bigoplus}_{\lambda\in\Lambda(m|n,d)}[\mathrm{xy}]_\lambda\mathcal{H}\big)
\cong\mathrm{End}_{\mathcal{H}}(V^{\otimes d}_{m|n}).
\end{align*}
\end{cor}

\section{Multiplication formulas}

\subsection{Quantum combinatorics}
For $a, b\in\mathbb{N}$ with $a\geq b$, denote
\begin{equation*}
[a]_{\mathrm{p}} = \frac{\mathrm{p}^{a}-1}{\mathrm{p}-1},\quad [a]_{\mathrm{p}}^! = [a]_{\mathrm{p}}[a-1]_{\mathrm{p}}\cdots[1]_{\mathrm{p}},\quad
\left[\begin{array}{cc}a\\b\end{array}\right]_{\mathrm{p}}=\frac{[a]_{\mathrm{p}}^!}{[b]_{\mathrm{p}}^![a-b]_{\mathrm{p}}^!},
\end{equation*}
where $\mathrm{p}\in\{q, q^{-1}\}$. It is understood that $[0]^!_{\mathrm{p}}=1$.

Furthermore, we denote
\begin{equation*}
[2a]_{\fc}=[a]_{q}(q^a+1),\quad [a]_{\fc}^! = \prod\limits_{k=1}^a[2k]_{\fc}.
\end{equation*}
For $A\in\Xi_{m|n,d}$, let
$$[A]_\fc^!=[\frac{a_{00}-1}{2}]_{\fc}^! \prod_{(i,j)\in I_\fa}[a_{ij}]_{\mathbf{q}_j}^!,\qquad\mbox{where}$$
$$\mathbf{q}_j:=q^{(-1)^{\widehat{j}}}\quad \mbox{and}\quad I_\fa:=(\{0\}\times[1..m+n])\sqcup([1..m+n]\times\mathbb{I}).$$

\begin{lem}
For any $A = \kappa(\mu,g,\nu)\in\Xi_{m|n, d}$,
we have
\begin{equation}\label{eq:Ac}
[A]_\fc^!=\sum_{w^{(0)}w^{(1)} \in W_{\delta(A)}} q^{\ell(w^{(0)})-\ell(w^{(1)})}.
\end{equation}
\end{lem}
\begin{proof}
Write $\delta=\delta(A)$.
We know
$W_{\delta}\cong W^{\fc}_{\frac{a_{00}-1}{2}} \times \prod_{(i,j)\in I_\fa} \mathfrak{S}_{a_{ij}}$, where $W^{\fc}_{k}$
(resp. $\mathfrak{S}_{l}$) is the Weyl group of type $C_{k}$ (resp. $A_{l-1}$). Their Poincare polynomials say that
\begin{equation*}
\sum_{w\in W^{\fc}_{k}}q^{\ell(w)}=\prod_{i=1}^{k}[2i]_{\fc}=[k]_{\fc}^!\quad
\mbox{and}\quad
\sum_{w\in \mathfrak{S}_{l}}{\mathrm{p}}^{\ell(w)}=\prod_{i=1}^{l}[i]_{\mathrm{p}}=[l]_{\mathrm{p}}^!,\quad(\mathrm{p}\in\{q, q^{-1}\}).
\end{equation*}
Note that
$W_{\delta^{(0)}}\cong W^{\fc}_{\frac{a_{00}-1}{2}} \times \prod\limits_{\tiny\substack{(i,j)\in I_\fa,\\j\leq m}} \mathfrak{S}_{a_{ij}}$ and $W_{\delta^{(1)}}\cong \prod\limits_{\tiny\substack{(i,j)\in I_\fa,\\j>m}} \mathfrak{S}_{a_{ij}}$. Therefore
\begin{align*}
 \sum_{w^{(0)}w^{(1)} \in W_{\delta}} q^{\ell(w^{(0)})-\ell(w^{(1)})}&=(\sum_{w\in W^{\fc}_{\frac{a_{00}-1}{2}}}q^{\ell(w)})\prod_{\tiny\substack{(i,j)\in I_\fa,\\j\leq m}}(\sum_{w\in \mathfrak{S}_{a_{ij}}}q^{\ell(w)})\prod_{\tiny\substack{(i,j)\in I_\fa,\\j>m}}(\sum_{w\in \mathfrak{S}_{a_{ij}}}q^{-\ell(w)})\\
 &=[\frac{a_{00}-1}{2}]_{\fc}^!\prod_{\tiny\substack{(i,j)\in I_\fa,\\j\leq m}}[a_{ij}]_q^!\prod_{\tiny\substack{(i,j)\in I_\fa,\\j> m}}[a_{ij}]_{q^{-1}}^!=[A]_\fc^!
\end{align*} as desired.
\end{proof}

\begin{lem}\label{lem:xTx}
For any $A = \kappa(\mu,g,\nu)\in\Xi_{m|n,d}$,
we have $\mathrm{[xy]}_{\mu} T_g \mathrm{[xy]}_{\nu} = [A]_\fc^! e_{A}(\mathrm{[xy]}_{\nu})$.
\end{lem}
\begin{proof}
Write $\delta=\delta(A)$. Thanks to Lemma~\ref{lem:doublecoset}(2) and Lemma~\ref{lem:comp}, we compute
\begin{align*}
\mathrm{[xy]}_{\nu} &= \sum_{w^{(0)}w^{(1)}\in W_{\nu}}T_{w^{(0)}}((-q)^{-\ell(w^{(1)})}T_{w^{(1)}})\\
&= (\sum_{w^{(0)}w^{(1)}\in W_{\delta}}T_{w^{(0)}}((-q)^{-\ell(w^{(1)})}T_{w^{(1)}}))(\sum_{w^{(0)}w^{(1)}\in \D_{\delta}\bigcap W_{\nu}}T_{w^{(0)}}((-q)^{-\ell(w^{(1)})}T_{w^{(1)}}))\\
&= (\sum_{w^{(0)}w^{(1)}\in W_{\delta}}T_{w^{(0)}}((-q)^{-\ell(w^{(1)})}T_{w^{(1)}})) T_{\D_{\delta}\cap W_{\nu}}.
\end{align*}
For any $w^{(0)}\in W_\delta^{(0)}=g^{-1}W_{\mu^{(0)}}g\cap W_{\nu^{(0)}}$ (cf. \eqref{eq:deltai}), there exists $w_{\mu^{(0)}}\in W_{\mu^{(0)}}$ such that $gw^{(0)} = w_{\mu^{(0)}}g$. It is clear that $\ell(gw^{(0)}) = \ell(g)+\ell(w^{(0)}) = \ell(w_{\mu^{(0)}})+\ell(g) = \ell(w_{\mu^{(0)}}g)$ because of $g\in\D_{\mu\nu}^{\circ}$. So $\ell(w^{(0)})=\ell(w_{\mu^{(0)}})$. Similarly, we know that for any $w^{(1)}\in W_\delta^{(1)}$, there exists $w_{\mu^{(1)}}\in W_{\mu^{(1)}}$ such that $gw^{(1)} = w_{\mu^{(1)}}g$ and $\ell(w^{(1)}) = \ell(w_{\mu^{(1)}})$.
Therefore
\begin{align*}
\mathrm{[xy]}_{\mu} T_g \mathrm{[xy]}_{\nu} &= \mathrm{[xy]}_{\mu} T_g (\sum_{w^{(0)}w^{(1)}\in W_{\delta}}T_{w^{(0)}}((-q)^{-\ell(w^{(1)})}T_{w^{(1)}})) T_{\D_{\delta}\cap W_{\nu}}\\
&= \mathrm{[xy]}_{\mu} (\sum_{w^{(0)}w^{(1)}\in W_{\delta}}T_{w_{\mu^{(0)}}}((-q)^{-\ell(w_{\mu^{(1)}})}T_{w_{\mu^{(1)}}})) T_g T_{\D_{\delta}\cap W_{\nu}}\\
&= (\sum_{w^{(0)}w^{(1)}\in W_{\delta}} q^{\ell(w_{\mu^{(0)}}) -\ell(w_{\mu^{(1)}})})\mathrm{[xy]}_{\mu}T_g T_{\D_{\delta}\cap W_{\nu}} \qquad \mbox{(by Lemma~\ref{lem:wx})}\\
&= (\sum_{w^{(0)}w^{(1)}\in W_{\delta}} q^{\ell(w^{(0)}) -\ell(w^{(1)})})\mathrm{[xy]}_{\mu}T_g T_{\D_{\delta}\cap W_{\nu}}\\
&= [A]_\fc^! T_{W_{\mu}gW_{\nu}} \qquad\qquad \mbox{(by \eqref{def:T} and \eqref{eq:Ac})}\\
&= [A]_\fc^! e_{A}(\mathrm{[xy]}_{\nu}).
\end{align*}
The proof is completed.
\end{proof}

\subsection{Multiplication formulas}
\begin{lem}\label{lem:e_Be_A}
Let $B = \kappa(\ld,g_1,\mu),A = \kappa(\mu,g_2,\nu)\in\Xi_{m|n,d}$ and $\delta = \delta(B)$,
we have
\begin{align*}
e_Be_A(\mathrm{[xy]}_{\nu}) = \frac{1}{[A]^!_\fc}\mathrm{[xy]}_{\ld} T_{g_1} (\sum_{w=w^{(0)}w^{(1)}\in \D_{\delta}\cap W_{\mu}}(-q)^{-\ell(w^{(1)})}T_{wg_2})  \mathrm{[xy]}_{\nu}.
\end{align*}
\end{lem}
\begin{proof}
It is a straightforward calculation that
\begin{align*}
e_Be_A(\mathrm{[xy]}_{\nu}) &= e_B(\frac{1}{[A]^!_\fc}\mathrm{[xy]}_{\mu} T_{g_2} \mathrm{[xy]}_{\nu})
\qquad\qquad\qquad \mbox{(by Lemma~\ref{lem:xTx})}\\
&=\frac{1}{[A]^!_\fc}(\mathrm{[xy]}_{\ld} T_{g_1} \sum_{w=w^{(0)}w^{(1)}\in \D_{\delta}\cap W_{\mu}}(-q)^{-\ell(w^{(1)})} T_{w}) T_{g_2} \mathrm{[xy]}_{\nu}\\
&\qquad\qquad\qquad\qquad \quad\qquad\qquad\qquad\qquad
\mbox{(by \eqref{def:T}, \eqref{def:phimunu})}\\
&=\frac{1}{[A]^!_\fc}\mathrm{[xy]}_{\ld} T_{g_1} (\sum_{w=w^{(0)}w^{(1)}\in \D_{\delta}\cap W_{\mu}}(-q)^{-\ell(w^{(1)})}T_{wg_2})  \mathrm{[xy]}_{\nu}\\
&\qquad\qquad\qquad\qquad \quad\qquad\qquad \mbox{(because $g_2\in\D_{\mu},w\in W_{\mu}$)}.
\end{align*}
\end{proof}

For any $A\in\Xi_{m|n,d}$, $h\in[0..m+n-1]$ and $k\in\mathbb{I}$, let
\begin{align*}
  A_{hk}=A+E^\theta_{hk}-E^{\theta}_{h+1,k};\qquad
  \check{A}_{hk}=A-E^{\theta}_{hk}+E^{\theta}_{h+1,k},
\end{align*} where $E_{ij}\in \mathrm{Mat}_{\mathbb{I}\times \mathbb{I}}(\mathbb{N})$ is the matrix whose $(i,j)$-th and $(-i,-j)$-th entries are $1$ and other entries are $0$.
Note that $A_{hk}$ and $\check{A}_{hk}$ may no longer belong to $\Xi_{m|n,d}$. In the following proposition, it is understood that $e_{A_{hk}}=0$ (resp. $e_{\check{A}_{hk}}=0$) if $A_{hk}\not\in\Xi_{m|n,d}$ (resp. $\check{A}_{hk}\not\in\Xi_{m|n,d}$).

For $0\leq h\leq m+n$, set
$
\bc{\mathbf{\dot{q}}_{h}=1,\quad\quad\,\,\, \mathbf{\ddot{q}}_{h}=q, &\mbox{if $\widehat{h}=0$};\\
\mathbf{\dot{q}}_{h}=-q^{-1},\quad \mathbf{\ddot{q}}_{h}=-1, &\mbox{if $\widehat{h}=1$}.}
$

Now we can give some more explicit multiplication formulas.
\begin{prop}\label{lem:e_Be_A}
Let $A=(a_{ij})_{i,j\in\mathbb{I}}=\kappa(\mu,g,\nu), B=\kappa(\ld,\id,\mu), C=\kappa(\ld',\id,\mu)\in\Xi_{m|n,d}$ and $h\in[0..m+n-1]$.
\begin{enumerate}
\item[(1)] Assume $B-E^{\theta}_{h,h+1}$ is diagonal.
For $h\neq0$, we have
\begin{equation*}
e_Be_A= \sum\limits_{k\in\mathbb{I}}{\mathbf{\dot{q}}_{h+1}}^{\sum_{j<k}a_{h+1,j}}
{\mathbf{\ddot{q}}_{h}}^{\sum_{j>k}a_{hj}} [a_{hk}+1]_{\mathbf{q}_{h}} e_{A_{hk}};
\end{equation*}
for $h=0$, we have
\begin{equation*}
e_Be_A= {\mathbf{\dot{q}}_{h+1}}^{\sum_{j<0}a_{1j}}{\mathbf{\ddot{q}}_{h}}^{\sum_{j>0}a_{0j}} [a_{00}+1]_\fc e_{A_{00}}+
\sum_{k\neq 0}{\mathbf{\dot{q}}_{h+1}}^{\sum_{j<k}a_{1j}}{\mathbf{\ddot{q}}_{h}}^{\sum_{j>k}a_{0j}} [a_{0k}+1]_{\mathbf{q}_{h}} e_{A_{0k}}.
\end{equation*}
\item[(2)] Assume $C-E^{\theta}_{h+1,h}(=C-E^{\theta}_{-h-1,-h})$ is diagonal. We have
\begin{equation*}
e_Ce_A= \sum\limits_{k\in\mathbb{I}}{\mathbf{\dot{q}}_{h}}^{\sum_{j>k}a_{hj}}{\mathbf{\ddot{q}}_{h+1}}^{\sum_{j<k}a_{h+1,j}} [a_{h+1,k}+1]_{\mathbf{q}_{h+1}} e_{\check{A}_{hk}}.
\end{equation*}
\end{enumerate}
\end{prop}
\begin{proof}
  We only prove the first item here, while the second one is similar.

  Write $\mu=(\mu^{(0)}|\mu^{(1)})=\mathrm{row}(A)=(\mu_0,\mu_1,\ldots,\mu_{m}|\mu_{m+1},\ldots,\mu_{m+n})$ and let $\widetilde{\mu}_h=\mu_0+\mu_1+\cdots+\mu_h$. Denote $\delta=\delta(B)$. It is clear that
  $$\mathcal{D}_\delta\cap W_\mu=\{1,s_{\widetilde{\mu}_h+1},s_{\widetilde{\mu}_h+1}s_{\widetilde{\mu}_h+2},\ldots,
  s_{\widetilde{\mu}_h+1}s_{\widetilde{\mu}_h+2}\cdots s_{\widetilde{\mu}_h+\mu_{h+1}-1}\}.$$
  Thus for any $w=w^{(0)}w^{(1)}\in\mathcal{D}_\delta\cap W_\mu$, either $w=w^{(0)}$ if $h\in[0..m-1]$ or $w=w^{(1)}$ if $h\in[m..m+n-1]$. Therefore $(-q)^{-\ell(w^{(1)})}=({\mathbf{\dot{q}}_{h+1}})^{\ell(w)}$.

  For any $w\in\D_{\delta}\cap W_{\mu}$, let $\mathrm{y}^w$ be the minimal length double coset representative in $W_{\ld}wgW_{\nu}$, and let $A^w=\kappa(\ld,\mathrm{y}^w,\nu)$.
There exists $w_{\ld}\in W_{\ld}$ and $w_{\nu}\in W_{\nu}$ such that $wg = w_{\ld}\mathrm{y}^ww_{\nu}$.
Note that $w_{\ld}$ either belongs to $W_{\ld^{(0)}}$ or to $W_{\ld^{(1)}}$ and $w_{\nu}$ either belongs to $W_{\nu^{(0)}}$ or to $W_{\nu^{(1)}}$.
Furthermore, we have
\begin{align*}
\ell(wg)=\ell(w)+\ell(g)=\ell(w_{\ld})+\ell(\mathrm{y}^w)+\ell(w_{\nu}).
\end{align*}
Therefore
\begin{align*}
e_Be_A(\textup{[xy]}_{\nu})
=&\frac{1}{[A]^!_\fc} (\sum\limits_{w=w^{(0)}w^{(1)}\in \D_{\delta}\cap W_{\mu}}(-q)^{-\ell(w^{(1)})}\mathrm{[xy]}_{\ld}T_{wg}\mathrm{[xy]}_{\nu})\\
=&\sum\limits_{w\in \D_{\delta}\cap W_{\mu}}\frac{1}{[A]^!_\fc}({\mathbf{\dot{q}}_{h+1}})^{\ell(w)}\mathrm{[xy]}_{\ld}T_{wg}\mathrm{[xy]}_{\nu}\\
=&\sum\limits_{w\in \D_{\delta}\cap W_{\mu}}\frac{1}{[A]^!_\fc}({\mathbf{\dot{q}}_{h+1}})^{\ell(w)}\mathrm{[xy]}_{\ld}T_{w_{\ld}} T_{\mathrm{y}^w} T_{w_{\nu}}\mathrm{[xy]}_{\nu}\\
=&\sum\limits_{w\in \D_{\delta}\cap W_{\mu}}\frac{1}{[A]^!_\fc}({\mathbf{\dot{q}}_{h+1}})^{\ell(w)}(\mathrm{[xy]}_{\ld}T_{w_{\ld}}) T_{\mathrm{y}^w} (T_{w_{\nu}}\mathrm{[xy]}_{\nu}).
\end{align*}

In order to write the above formula in a more explicit form, let us analyze the elements in $\D_{\delta}\cap W_{\mu}$ in detail.
For some $a_{h+1,k}>0$ and $0\leq p < a_{h+1,k}$, denote
\begin{equation}
\theta=s_{\widetilde{\mu}_{h}+1}s_{\widetilde{\mu}_{h}+2}\cdots s_{\widetilde{\mu}_{h}+\sum_{j=1}^{k-1}a_{h+1,j}},\quad \theta_p=s_{\widetilde{\mu}_{h}+\sum_{j=1}^{k-1}a_{h+1,j}+1}\cdots s_{\widetilde{\mu}_{h}+\sum_{j=1}^{k-1}a_{h+1,j}+p},
\end{equation} and let
$w_p = \theta\theta_p$.
It is understood that $\theta_0=1$. We have $$\ell(w_p)=\ell(\theta)+p=\ell(w_0)+p=\sum_{j<k}a_{h+1,j}+p$$
and $w_p\in\D_{\delta}\cap W_{\mu}$. 
Thus there exists $w_{\ld_p}\in W_{\ld}$ and $w_{\nu_p}\in W_{\nu}$ such that $w_pg = w_{\ld_p}\mathrm{y}^{w_p}w_{\nu_p}$, and  $\ell(w_{\ld_p})=\sum_{j>k}a_{h,j}$,
$\ell(w_{\nu_p})=p$.
The matrices $A^{w_p}$ corresponding to the elements $w_p, (0\leq p < a_{h+1,k})$, are all equal to $A_{hk}$ due to the invariance of $|R^\ld_i\cap w_pgR^\mu_j|$. Hence $\mathrm{y}^{w_p}\,\,(0\leq p < a_{h+1,k})$ are also all the same (we rewritten them by $\mathrm{y}_{hk}$).

We continue to compute the formula.

\noindent
Case 1: $\widehat{h+1}=\widehat{k}$. \\
Note we have known $w_p g=w_{\ld_p}\textup{y}_{hk}w_{\nu_p}$,
$\ell(w_p)=\sum_{j<k}a_{h+1,j}+p$, $\ell(w_{\ld_p})=\sum_{j>k}a_{hj}$, $\ell(w_{\nu_p})=p$. Therefore
\begin{align}\label{eq:even0}
&\nonumber\sum\limits_{p=0}^{a_{h+1,k}-1}\ds\frac{1}{[A]^!_\fc}({\mathbf{\dot{q}}_{h+1}})^{\ell(w_p)}(\mathrm{[xy]}_{\ld}T_{w_{\ld_p}}) T_{\mathrm{y}_{hk}} (T_{w_{\nu_p}}\mathrm{[xy]}_{\nu})\\
=&\nonumber\sum\limits_{p=0}^{a_{h+1,k}-1}\ds\frac{1}{[A]^!_\fc}{\mathbf{\dot{q}}_{h+1}}^{\sum\limits_{j<k}a_{h+1,j}}(\mathrm{[xy]}_{\ld}T_{w_{\ld_p}}) T_{\mathrm{y}_{hk}} (\mathbf{\dot{q}}_{h+1}^p T_{w_{\nu_p}}\mathrm{[xy]}_{\nu})\\
=&\nonumber\sum\limits_{p=0}^{a_{h+1,k}-1}\ds\frac{1}{[A]^!_\fc}{\mathbf{\dot{q}}_{h+1}}^{\sum\limits_{j<k}a_{h+1,j}}({\mathbf{\ddot{q}}_{h}}^{\sum\limits_{j>k}a_{hj}}\mathrm{[xy]}_{\ld}) T_{\mathrm{y}_{hk}} (\mathbf{\dot{q}}_{h+1}^p \mathbf{\ddot{q}}_{k}^p \mathrm{[xy]}_{\nu})\\
=~&\nonumber\ds\frac{1}{[A]^!_\fc}{\mathbf{\dot{q}}_{h+1}}^{\sum\limits_{j<k}a_{h+1,j}}{\mathbf{\ddot{q}}_{h}}^{\sum\limits_{j>k}a_{hj}}
(1+{\mathbf{\dot{q}}_{h+1}\mathbf{\ddot{q}}_{k}}+\cdots+({\mathbf{\dot{q}}_{h+1}\mathbf{\ddot{q}}_{k}})^{a_{h+1,k}-1})
\mathrm{[xy]}_{\ld} T_{\mathrm{y}_{hk}} \mathrm{[xy]}_{\nu}\\
=~&\nonumber\ds\frac{[A_{hk}]^!_\fc}{[A]^!_\fc}{\mathbf{\dot{q}}_{h+1}}^{\sum\limits_{j<k}a_{h+1,j}}{\mathbf{\ddot{q}}_{h}}^{\sum\limits_{j>k}a_{hj}}
(1+{\mathbf{\dot{q}}_{h+1}\mathbf{\ddot{q}}_{k}}+\cdots+({\mathbf{\dot{q}}_{h+1}\mathbf{\ddot{q}}_{k}})^{a_{h+1,k}-1})e_{A_{hk}}(\mathrm{[xy]}_{\nu})\\
=~&\ds\frac{[A_{hk}]^!_\fc[a_{h+1,k}]_{\mathbf{\dot{q}}_{h+1}\mathbf{\ddot{q}}_{k}}}{[A]^!_\fc}
{\mathbf{\dot{q}}_{h+1}}^{\sum\limits_{j<k}a_{h+1,j}}{\mathbf{\ddot{q}}_{h}}^{\sum\limits_{j>k}a_{hj}}e_{A_{hk}}(\mathrm{[xy]}_{\nu}).
\end{align}

\noindent
Case 2: $\widehat{h+1}\neq\widehat{k}$.\\
In this case, $a_{h+1,k}\in\{0,1\}$.
We always assume $a_{h+1,k}=1$ since $a_{h+1,k}=0$ has no meaning for discussion. We get $w_p=\theta$ since $\theta_p=1$ when $a_{h+1,k}=1$.
Thus $w_p g=w_{\ld_p}\textup{y}_{hk}$ and $\ell(w_p)=\sum\limits_{j<k}a_{h+1,j},\,\,\ell(w_{\ld_p})=\sum\limits_{j>k}a_{hj}$. Then
\begin{align}\label{eq:odd1}
&\nonumber\sum\limits_{p=0}^{a_{h+1,k}-1}\ds\frac{1}{[A]^!_\fc}({\mathbf{\dot{q}}_{h+1}})^{\ell(w_p)}(\mathrm{[xy]}_{\ld}T_{w_{\ld_p}}) T_{\mathrm{y}_{hk}} (T_{w_{\nu_p}}\mathrm{[xy]}_{\nu})
\\
=&\nonumber\ds\frac{1}{[A]^!_\fc}{\mathbf{\dot{q}}_{h+1}}^{\sum\limits_{j<k}a_{h+1,j}}{\mathbf{\ddot{q}}_{h}}^{\sum\limits_{j>k}a_{hj}}\mathrm{[xy]}_{\ld} T_{\mathrm{y}_{hk}} \mathrm{[xy]}_{\nu}
\\
=&\ds\frac{[A_{hk}]^!_\fc}{[A]^!_\fc}{\mathbf{\dot{q}}_{h+1}}^{\sum\limits_{j<k}a_{h+1,j}}{\mathbf{\ddot{q}}_{h}}^{\sum\limits_{j>k}a_{hj}}e_{A_{hk}}(\mathrm{[xy]}_{\nu}).
\end{align}

Now let us simplify the coefficients of ${\mathbf{\dot{q}}_{h+1}}^{\sum\limits_{j<k}a_{h+1,j}}{\mathbf{\ddot{q}}_{h}}^{\sum\limits_{j>k}a_{hj}}e_{A_{hk}}(\textup{[xy]}_{\nu})$
in \eqref{eq:even0} and \eqref{eq:odd1}, respectively.

\noindent
Case I: $h\neq m$ or $0$ (hence $\widehat{h}=\widehat{h+1}$).\\
If $\widehat{h}=\widehat{k}$ then
${\mathbf{q}}_{h}={\mathbf{q}}_{k}$ and ${\mathbf{\dot{q}}_{h+1}\mathbf{\ddot{q}}_{k}}={\mathbf{q}}_{h}$. Hence the factor $\frac{[A_{hk}]^!_\fc[a_{h+1,k}]_{\mathbf{\dot{q}}_{h+1}\mathbf{\ddot{q}}_{k}}}{[A]^!_\fc}$, appearing in \eqref{eq:even0}, becomes to
\begin{align*}
\ds\frac{[A_{hk}]^!_\fc[a_{h+1,k}]_{\mathbf{\dot{q}}_{h+1}\mathbf{\ddot{q}}_{k}}}{[A]^!_\fc}
=\ds\frac{[a_{hk}+1]_{\mathbf{q}_{k}}[a_{h+1,k}]_{\mathbf{\dot{q}}_{h+1}\mathbf{\ddot{q}_{k}}}}{[a_{h+1,k}]_{\mathbf{q}_{k}}}
=[a_{hk}+1]_{\mathbf{q}_{h}}.
\end{align*}
If $\widehat{h}\neq\widehat{k}$ then $a_{hk}+1=a_{h+1,k}=1$
and hence $[a_{hk}+1]_{\mathbf{q}_{k}}=[a_{h+1,k}]_{\mathbf{q}_{k}}=1$. So the factor $\frac{[A_{hk}]^!_\fc}{[A]^!_\fc}$, appearing in \eqref{eq:odd1}, becomes to
\begin{align*}
\ds\frac{[A_{hk}]^!_\fc}{[A]^!_\fc}=\ds\frac{[a_{hk}+1]_{\mathbf{q}_{k}}}{[a_{h+1,k}]_{\mathbf{q}_{k}}}=1=[a_{hk}+1]_{\mathbf{q}_{h}}.
\end{align*}

\noindent
Case II: $h=m\neq0$ (hence $\widehat{h}=0$ and $\widehat{h+1}=1$ ). \\
If $\widehat{k}=1$ (hence $\widehat{h+1}=\widehat{k}$), then ${\mathbf{q}_{k}}={\mathbf{\dot{q}}_{m+1}}{\mathbf{\ddot{q}}_{k}}=q^{-1}$
since ${\mathbf{q}}_{k}=q^{-1}$, ${\mathbf{\dot{q}}_{m+1}}=-q^{-1}$, ${\mathbf{\ddot{q}}_{k}}=-1$. Moreover, $a_{hk}=a_{mk}=0$ (otherwise $a_{mk}+1\not\in\{0,1\}$ is invalid since $\widehat{m}\neq\widehat{k}$). Now the factor $\frac{[A_{hk}]^!_\fc[a_{h+1,k}]_{\mathbf{\dot{q}}_{h+1}\mathbf{\ddot{q}}_{k}}}{[A]^!_\fc}$ appearing in \eqref{eq:even0} becomes to
\begin{align*}
\ds\frac{[A_{mk}]^!_\fc[a_{m+1,k}]_{\mathbf{\dot{q}}_{m+1}\mathbf{\ddot{q}}_{k}}}{[A]^!_\fc}
=\ds\frac{[a_{mk}+1]_{\mathbf{q}_{k}}[a_{m+1,k}]_{\mathbf{q}_{k}}}{[a_{m+1,k}]_{\mathbf{q}_{k}}}
=[a_{mk}+1]_{\mathbf{q}_{k}}=1=[a_{hk}+1]_{\mathbf{q}_{h}}.
\end{align*}
If $\widehat{k}=0$ (hence $\widehat{h+1}\neq\widehat{k}$) then ${\mathbf{q}}_{m}={\mathbf{q}}_{k}=q$ and $a_{m+1,k}=1$ (hence $[a_{m+1,k}]_{\mathbf{q}_{k}}=1$). The factor $\frac{[A_{hk}]^!_\fc}{[A]^!_\fc}$ appearing in \eqref{eq:odd1} becomes to
\begin{align*}
\ds\frac{[A_{mk}]^!_\fc}{[A]^!_\fc}=
\ds\frac{[a_{mk}+1]_{\mathbf{q}_{k}}}{[a_{m+1,k}]_{\mathbf{q}_{k}}}
=[a_{mk}+1]_{\mathbf{q}_{k}}=[a_{mk}+1]_{\mathbf{q}_{m}}=[a_{hk}+1]_{\mathbf{q}_{h}}.
\end{align*}

\noindent
Case III: $0=h<m$ (hence $\widehat{h}=\widehat{h+1}=0$ ). \\
If $\widehat{k}=0=\widehat{h+1}$ then ${\mathbf{q}}_{h}={\mathbf{q}}_{k}$
and ${\mathbf{\dot{q}}_{h+1}\mathbf{\ddot{q}}_{k}}={\mathbf{q}}_{h}$. The factor $\frac{[A_{hk}]^!_\fc[a_{h+1,k}]_{\mathbf{\dot{q}}_{h+1}\mathbf{\ddot{q}}_{k}}}{[A]^!_\fc}$ in \eqref{eq:even0} equals to
\begin{align*}
\ds\frac{[A_{0k}]^!_\fc[a_{1k}]_{\mathbf{\dot{q}}_{h+1}\mathbf{\ddot{q}}_{k}}}{[A]^!_\fc}=
\bc{\ds\frac{[a_{0k}+1]_{\mathbf{q}_k}[a_{1k}]_{\mathbf{q}_h}}{[a_{1k}]_{\mathbf{q}_k}}=[a_{0k}+1]_{\mathbf{q}_{h}}, &\mbox{if $k\neq0$};\\
\ds\frac{[2(\frac{a_{00}-1}{2}+1)]_{\fc}[a_{10}]_{\mathbf{q}_{h}}}{[a_{10}]_{\mathbf{q}_{k}}}=[a_{00}+1]_{\fc}, &\mbox{if $k=0$}.}
\end{align*}
If $\widehat{k}=1\neq\widehat{h+1}$, we have $a_{0k}+1=a_{1k}=1$. The factor $\frac{[A_{hk}]^!_\fc}{[A]^!_\fc}$ appearing in \eqref{eq:odd1} becomes to
\begin{align*}
\ds\frac{[A_{0k}]^!_\fc}{[A]^!_\fc}=\ds\frac{[a_{0k}+1]_{\mathbf{q}_{k}}}{[a_{1k}]_{\mathbf{q}_{k}}}=1=[a_{0k}+1]_{\mathbf{q}_{h}}.
\end{align*}

\noindent
Case IV: $h=m=0$ (hence $0=\widehat{h}\neq\widehat{h+1}=1$ ). \\
If $k\neq0$ so that $\widehat{k}=1=\widehat{h+1}$, then
$a_{0k}+1=1$ because of $\widehat{k}=1\neq\widehat{0}$, and ${\mathbf{q}_{k}}={\mathbf{\dot{q}}_{m+1}}{\mathbf{\ddot{q}}_{k}}=q^{-1}$
since ${\mathbf{q}}_{k}=q^{-1}$, ${\mathbf{\dot{q}}_{m+1}}=-q^{-1}$, ${\mathbf{\ddot{q}}_{k}}=-1$.
The factor $\frac{[A_{hk}]^!_\fc[a_{h+1,k}]_{\mathbf{\dot{q}}_{h+1}\mathbf{\ddot{q}}_{k}}}{[A]^!_\fc}$ in \eqref{eq:even0} equals to
\begin{align*}
\ds\frac{[A_{0k}]^!_\fc[a_{1k}]_{\mathbf{\dot{q}}_{h+1}\mathbf{\ddot{q}}_{k}}}{[A]^!_\fc}
=\ds\frac{[a_{0k}+1]_{\mathbf{q}_{k}}[a_{1k}]_{\mathbf{q}_{k}}}{[a_{1k}]_{\mathbf{q}_{k}}}
=[a_{0k}+1]_{\mathbf{q}_{k}}=1=[a_{0k}+1]_{\mathbf{q}_{h}}.
\end{align*}
If $k=0$ so that $\widehat{k}=0\neq\widehat{h+1}$, then $a_{h+1,k}=a_{10}=1$ and ${\mathbf{q}}_{m}={\mathbf{q}}_{k}=q$. The factor $\frac{[A_{hk}]^!_\fc}{[A]^!_\fc}$ appearing in \eqref{eq:odd1} becomes to
\begin{align*}
\ds\frac{[A_{00}]^!_\fc}{[A]^!_\fc}=\ds\frac{[2(\frac{a_{00}-1}{2}+1)]_{\fc}}{[a_{10}]_{\mathbf{q}_{k}}}=[a_{00}+1]_{\fc}.
\end{align*}

The proposition is proved.
\end{proof}

We call $A\in\Xi_{m|n, d}$ a \emph{Chevalley matrix} if either $A-aE_{h,h+1}^{\theta}$ or $A-aE_{h+1,h}^{\theta}$ is diagonal for some $a\in\mathbb{N}$ and $h\in[0..m+n-1]$.
We note here that if $B-bE^{\theta}_{m,m+1}$ and $C-cE^{\theta}_{m+1,m}$ are diagonal for Chevalley matrices $B,C\in\Xi_{m|n,d}$ then it must be $b,c\in\{0,1\}$ because of the even-odd trivial intersection property. We shall not list the formulas for $h=m$ in the following proposition since this case has been done in Proposition~\ref{lem:e_Be_A}.
\begin{prop}\label{thm:e_Bpe_A}
Let $A,B,C\in\Xi_{m|n, d}$ and $m\neq h\in[0..m+n-1]$.
\begin{enumerate}
\item[(1)] Assume $B-bE_{h,h+1}^{\theta}$ is diagonal and $\co(B)=\ro(A)$.
If $h\neq0$, we have
\begin{equation}\label{e_Bbe_A}
e_Be_A= \sum_{\mathbf{t}}{\mathbf{\dot{q}}_{h+1}}^{\sum_{j<k}(a_{h+1,j}-t_j)t_k}{\mathbf{\ddot{q}}_{h}}^{\sum_{j>k}a_{hj}t_k} \prod_{k\in\mathbb{I}}\left[\begin{array}{cc}a_{hk}+t_k\\t_k\end{array}\right]_{\mathbf{q}_{h}} e_{A_{\mathbf{t},h+1}},
\end{equation}
where $\mathbf{t}=(t_i)_{i\in\mathbb{I}}\in\mathbb{N}^{N}$ satisfies $\sum_{i\in\mathbb{I}}t_{i}=b$ and $t_{i}\leq a_{h+1, i}$; and $A_{\mathbf{t},h+1}= A+\sum_{i\in\mathbb{I}}t_{i}(E^{\theta}_{hi}-E^{\theta}_{h+1,i})$.

If $h=0$, we have
\begin{equation}\label{e_B0e_A}
\begin{split}
e_Be_A
=&\sum_{\mathbf{t}}
{\mathbf{\dot{q}}_{h+1}}^{\sum_{j<k}(a_{1j}-t_j)t_k}
{\mathbf{\ddot{q}}_{h}}^{\sum_{j>k}a_{0j}t_k+\sum\limits_{k<j<-k}t_j t_k+\frac{1}{2}\sum_{k<0}t_k(t_k-1)}
\\
&\cdot\frac{[\frac{a_{00}-1}{2}+t_0]_\fc^!}{[\frac{a_{00}-1}{2}]_\fc^![t_0]_{\mathbf{q}_{h}}^!}
\prod_{k=1}^{m+n}\frac{[a_{0k}+t_k+t_{-k}]_{\mathbf{q}_{h}}^!}{[a_{0k}]_{\mathbf{q}_{h}}^![t_k]_{\mathbf{q}_{h}}^![t_{-k}]_{\mathbf{q}_{h}}^!}e_{A_{\mathbf{t},1}},
\end{split}
\end{equation}
where $\mathbf{t}=(t_i)_{i\in\mathbb{I}}\in\mathbb{N}^{N}$ satisfies $\sum_{i\in\mathbb{I}}t_{i}=b$ and $t_{i}\leq a_{1i}$.

\item[(2)] Assume $C-cE_{h+1,h}^{\theta}(=C-cE_{-h-1,-h}^{\theta})$ is diagonal and $\co(C)=\ro(A)$. We have
\begin{equation}\label{e_Cce_A}
e_Ce_A= \sum_{\mathbf{t}}{\mathbf{\dot{q}}_{h}}^{\sum_{j>k}(a_{hj}-t_j)t_k}{\mathbf{\ddot{q}}_{h+1}}^{\sum_{j<k}a_{h+1,j}t_k} \prod_{k\in\mathbb{I}}\left[\begin{array}{cc}a_{h+1,k}+t_k\\t_k\end{array}\right]_{\mathbf{q}_{h+1}} e_{\check{A}_{\mathbf{t},h+1}},
\end{equation}
where $\mathbf{t}=(t_i)_{i\in\mathbb{I}} \in\mathbb{N}^{N}$ satisfies $$\sum_{i\in\mathbb{I}}t_{i}=c\quad \mbox{and}\quad \bc{
t_i\leq a_{hi} &\textup{if}\,\,h>0
\\
t_i+t_{-i}\leq a_{hi} &\textup{if}\,\,h=0
},$$ and $\check{A}_{\mathbf{t},h+1}= A+\sum\limits_{i\in\mathbb{I}}t_{i}(E^{\theta}_{h+1,i}-E^{\theta}_{hi})$.
\end{enumerate}
\end{prop}
\begin{proof}
 For the proof of $(1)$, let $\delta=\delta(B)$, $\mu=\co(B)$ and $\mathbf{t}=(t_i)_{i\in\mathbb{I}}\in\mathbb{N}^{N}$
 satisfies $\sum_{i\in\mathbb{I}}t_{i}=b$ and $t_{i}\leq a_{h+1, i}$.

Case 1: $h\neq0$.\\
 Similar to the argument for Proposition~\ref{lem:e_Be_A}, we focus on such elements $w\in\D_{\delta}\cap W_{\mu}$ that satisfy $A^w=A_{\mathbf{t},h+1}$.
 Among those $w$, there is a unique shortest element $w^\mathbf{t}$ with
$$\ell(w^\mathbf{t})=\sum\limits_{j<k}(a_{h+1,j}-t_j)t_k.$$
For $w\in\D_{\delta}\cap W_{\mu}$, there exists $w_{\ld}\in W_{\ld}$ and $w_{\nu}\in W_{\nu}$ such that $wg = w_{\ld}\mathrm{y}^{w}w_{\nu}$.
Moreover, we have $\ell(wg)=\ell(w)+\ell(g)=\ell(w_{\ld})+\ell(\mathrm{y}^{w})+\ell(w_{\nu})$. Taking $w=w^\mathbf{t}$, the lengths of corresponding $w^\mathbf{t}_\ld$ and $w^\mathbf{t}_\nu$ are
\begin{equation*}
\ell(w^\mathbf{t}_\ld)=\sum\limits_{j>k}a_{hj}t_k, \quad\ell(w^\mathbf{t}_{\nu})=0.
\end{equation*}
Thus
\begin{align*}
&\quad\sum_{\tiny\substack{w\in W_{\D_{\delta}\cap W_{\mu}},\\A^w=A_{\mathbf{t},h+1}}}\ds\frac{1}{[A]^!_\fc}({\mathbf{\dot{q}}_{h+1}})^{\ell(w)}\mathrm{[xy]}_{\ld}T_{wg}\mathrm{[xy]}_{\nu}\\
&=\sum_{\tiny\substack{w\in W_{\D_{\delta}\cap W_{\mu}},\\A^w=A_{\mathbf{t},h+1}}}\ds\frac{1}{[A]^!_\fc}({\mathbf{\dot{q}}_{h+1}})^{\ell(w)}
(\mathrm{[xy]}_{\ld}T_{w_{\ld}}) T_{\mathrm{y}^{w}} (T_{w_{\nu}}\mathrm{[xy]}_{\nu})\\
&=\ds\frac{1}{[A]^!_\fc}\prod_{k\in\mathbb{I}}\left[\begin{array}{cc}a_{h+1,k}\\t_k\end{array}\right]_{\mathbf{\dot{q}}_{h+1}\mathbf{\ddot{q}}_{k}} {\mathbf{\dot{q}}_{h+1}}^{\ell(w^\mathbf{t})}{\mathbf{\ddot{q}}_{h}}^{\ell(w^\mathbf{t}_\ld)}\mathrm{[xy]}_{\ld}T_{\mathrm{y}^{w^\mathbf{t}}}\mathrm{[xy]}_{\nu}\\
&=\ds\frac{[A^{w^\mathbf{t}}]^!_\fc}{[A]^!_\fc}\prod_{k\in\mathbb{I}}\left[\begin{array}{cc}a_{h+1,k}\\t_k\end{array}\right]_{\mathbf{\dot{q}}_{h+1}\mathbf{\ddot{q}}_{k}} {\mathbf{\dot{q}}_{h+1}}^{{\sum_{j<k}(a_{h+1,j}-t_j)t_k}}{\mathbf{\ddot{q}}_{h}}^{\sum_{j>k}a_{hj}t_k}e_{A_{\mathbf{t},h+1}}(\mathrm{[xy]}_{\nu}).
\end{align*}
All we have to do next is to prove that the following equation \eqref{eq:123456} holds.
\begin{equation}\label{eq:123456}
\ds\frac{[A^{w^\mathbf{t}}]^!_\fc}{[A]^!_\fc}\prod_{k\in\mathbb{I}}\left[\begin{array}{cc}a_{h+1,k}\\t_k\end{array}\right]_{\mathbf{\dot{q}}_{h+1}\mathbf{\ddot{q}}_{k}}
=\prod_{k\in\mathbb{I}}\left[\begin{array}{cc}a_{hk}+t_k\\t_k\end{array}\right]_{\mathbf{q}_{h}}.
\end{equation}

For the left hand side of the equation \eqref{eq:123456}, we have
 \begin{align*}
\textup{LHS}&=\prod_{k\in\mathbb{I}}
\frac{[a_{hk}+t_k]_{\mathbf{{q}}_{k}}^![a_{h+1,k}-t_k]_{\mathbf{{q}}_{k}}^![a_{h+1,k}]_{\mathbf{\dot{q}}_{h+1}\mathbf{\ddot{q}}_{k}}^!}{[a_{hk}]_{\mathbf{{q}}_{k}}^![a_{h+1,k}]_{\mathbf{{q}}_{k}}^![a_{h+1,k}-t_k]_{\mathbf{\dot{q}}_{h+1}\mathbf{\ddot{q}}_{k}}^![t_k]_{\mathbf{\dot{q}}_{h+1}\mathbf{\ddot{q}}_{k}}^!}.
\end{align*}
We only deal with the case of $h\in[1..m-1]$ here, while the case of $h\in[m+1..m+n-1]$ is similar.
For $k\in\mathbb{I}_{0}$, we have ${\mathbf{q}}_{h}={\mathbf{q}}_{k}$ and ${\mathbf{\dot{q}}_{h+1}\mathbf{\ddot{q}}_{k}}={\mathbf{q}}_{h}$, then
\begin{align*}
&\quad\prod_{k\in\mathbb{I}_{0}}
\ds\frac{[a_{hk}+t_k]_{\mathbf{{q}}_{k}}^![a_{h+1,k}-t_k]_{\mathbf{{q}}_{k}}^![a_{h+1,k}]_{\mathbf{\dot{q}}_{h+1}\mathbf{\ddot{q}}_{k}}^!}
{[a_{hk}]_{\mathbf{q}_{k}}^![a_{h+1,k}]_{\mathbf{q}_{k}}^![a_{h+1,k}-t_k]_{\mathbf{\dot{q}}_{h+1}\mathbf{\ddot{q}}_{k}}^![t_k]_{\mathbf{\dot{q}}_{h+1}\mathbf{\ddot{q}}_{k}}^!}\\
&=\prod_{k\in\mathbb{I}_{0}}\ds\frac{[a_{hk}+t_k]_{\mathbf{{q}}_{h}}^![a_{h+1,k}-t_k]_{\mathbf{{q}}_{h}}^![a_{h+1,k}]_{\mathbf{{q}}_{h}}^!}
{[a_{hk}]_{\mathbf{q}_{h}}^![a_{h+1,k}]_{\mathbf{q}_{h}}^![a_{h+1,k}-t_k]_{\mathbf{q}_{h}}^![t_k]_{\mathbf{q}_{h}}^!}\\
&=\prod_{k\in\mathbb{I}_{0}}\ds\frac{[a_{hk}+t_k]_{\mathbf{{q}}_{h}}^!}{[a_{hk}]_{\mathbf{q}_{h}}^![t_k]_{\mathbf{q}_{h}}^!}=\prod_{k\in\mathbb{I}_{0}} \left[\begin{array}{cc}a_{hk}+t_k\\t_k\end{array}\right]_{\mathbf{q}_{h}}.
\end{align*}
 For $k\in\mathbb{I}_{1}$, we have $t_k=1$ and $a_{hk}+t_k=a_{h+1,k}=1$, then
 \begin{align*}
&\prod_{k\in\mathbb{I}_{1}}
\ds\frac{[a_{hk}+t_k]_{\mathbf{{q}}_{k}}^![a_{h+1,k}-t_k]_{\mathbf{{q}}_{k}}^![a_{h+1,k}]_{\mathbf{\dot{q}}_{h+1}\mathbf{\ddot{q}}_{k}}^!}
{[a_{hk}]_{\mathbf{q}_{k}}^![a_{h+1,k}]_{\mathbf{q}_{k}}^![a_{h+1,k}-t_k]_{\mathbf{\dot{q}}_{h+1}\mathbf{\ddot{q}}_{k}}^![t_k]_{\mathbf{\dot{q}}_{h+1}\mathbf{\ddot{q}}_{k}}^!}\\
=&\prod_{k\in\mathbb{I}_{1}}
\ds\frac{[1]_{\mathbf{{q}}_{k}}^![0]_{\mathbf{{q}}_{k}}^![1]_{\mathbf{\dot{q}}_{h+1}\mathbf{\ddot{q}}_{k}}^!}
{[0]_{\mathbf{q}_{k}}^![1]_{\mathbf{q}_{k}}^![0]_{\mathbf{\dot{q}}_{h+1}\mathbf{\ddot{q}}_{k}}^![1]_{\mathbf{\dot{q}}_{h+1}\mathbf{\ddot{q}}_{k}}^!}
=\prod_{k\in\mathbb{I}_{1}} \left[\begin{array}{cc}a_{hk}+t_k\\t_k\end{array}\right]_{\mathbf{q}_{h}}.
\end{align*}
Therefore, the equation~\eqref{eq:123456} holds and so does the formula~\eqref{e_Bbe_A}.

  Case 2: $h=0$. \\
  Among those $w\in\D_\delta \cap W_\mu$ such that $A^{w}=A_{\mathbf{t},1}$, there is a unique shortest element $^tw$ with
$\ell(^\mathbf{t}w)=\sum_{j<k}(a_{1j}-t_j)t_k$, and the lengths of corresponding $^\mathbf{t} w_\ld$ and $^\mathbf{t} w_\nu$ are
\begin{equation*}\label{eq:lwldlwmu0}
\ell(^\mathbf{t} w_\ld)=
\sum_{j>k}
a_{0j}t_k+\sum_{k<j<-k}t_k t_j+\frac{1}{2}\sum_{k<0}t_k(t_k-1), \qquad\ell(^\mathbf{t}w_{\nu})=0.
\end{equation*}
Thus
\begin{align*}
&\sum_{\tiny\substack{w\in W_{\D_{\delta}\cap W_{\mu}},\\A^w=A_{\mathbf{t},1}}}\ds\frac{1}{[A]^!_\fc}({\mathbf{\dot{q}}_{h+1}})^{\ell(w)}\mathrm{[xy]}_{\ld}T_{wg}\mathrm{[xy]}_{\nu}
=\\
&\ds\frac{[A^{^\mathbf{t} w}]^!_\fc}{[A]^!_\fc}\prod_{k\in \mathbb{I}}\left[\begin{array}{cc}a_{1k}\\t_k\end{array}\right]_{\mathbf{\dot{q}}_{h+1}\mathbf{\ddot{q}}_{k}} {\mathbf{\dot{q}}_{h+1}}^{{\sum\limits_{j<k}(a_{h+1,j}-t_j)t_k}}{\mathbf{\ddot{q}}_{h}}^{\sum\limits_{j>k}
a_{0j}t_k+\sum\limits_{k<j<-k}t_j t_k+\frac{1}{2}\sum\limits_{k<0}t_k(t_k-1)}e_{A_{\mathbf{t},1}}(\mathrm{[xy]}_{\nu}).
\end{align*}
What we need to do next is to show the following equation \eqref{eq:1234560} holds.
\begin{equation}\label{eq:1234560}
\ds\frac{[A^{^t w}]^!_\fc}{[A]^!_\fc}\prod_{k\in \mathbb{I}}\left[\begin{array}{cc}a_{1k}\\t_k\end{array}\right]_{\mathbf{\dot{q}}_{h+1}\mathbf{\ddot{q}}_{k}}
=\frac{[\frac{a_{00}-1}{2}+t_0]_\fc^!}{[\frac{a_{00}-1}{2}]_\fc^![t_0]_{\mathbf{q}_{h}}^!}\prod_{k=1}^{m+n}\frac{[a_{0k}+t_k+t_{-k}]_{\mathbf{q}_{h}}^!}{[a_{0k}]_{\mathbf{q}_{h}}^![t_k]_{\mathbf{q}_{h}}^![t_{-k}]_{\mathbf{q}_{h}}^!}.
\end{equation}
The left hand side of the equation~\eqref{eq:1234560} is
\begin{equation*}\label{eq:12345600}
\textup{LHS}=\ds\frac{[\frac{a_{00}-1}{2}+t_0]^!_\fc \prod^{m+n}_{k=1}[a_{0k}+t_{k}+t_{-k}]_{\mathbf{q}_{k}}^! \prod_{k\in \mathbb{I}}[a_{1k}-t_{k}]^!_{\mathbf{q}_{k}}[a_{1k}]_{\mathbf{\dot{q}}_{h+1}\mathbf{\ddot{q}}_{k}}^!}{[\frac{a_{00}-1}{2}]^!_\fc \prod^{m+n}_{k=1 }[a_{0k}]_{\mathbf{q}_{k}}^!\prod_{k\in \mathbb{I}}[a_{1k}]^!_{\mathbf{q}_{k}}[a_{1k}-t_k]_{\mathbf{\dot{q}}_{h+1}\mathbf{\ddot{q}}_{k}}^![t_k]_{\mathbf{\dot{q}}_{h+1}\mathbf{\ddot{q}}_{k}}^!}.
\end{equation*}
We know that $h,\,\,h+1\leq m$ since $h\neq m$. For $k\in\mathbb{I}_{0}$, we have ${\mathbf{q}}_{h}={\mathbf{q}}_{k}$ and ${\mathbf{\dot{q}}_{h+1}\mathbf{\ddot{q}}_{k}}={\mathbf{q}}_{h}$. Hence
\begin{align*}
&\quad\ds\frac{[\frac{a_{00}-1}{2}+t_0]^!_\fc \prod^{m}_{k=1}[a_{0k}+t_{k}+t_{-k}]_{\mathbf{q}_{h}}^! \prod_{k\in\mathbb{I}_{0}}[a_{1k}-t_{k}]^!_{\mathbf{q}_{h}}[a_{1k}]_{\mathbf{q}_{h}}^!}{[\frac{a_{00}-1}{2}]^!_\fc \prod^{m}_{k=1}[a_{0k}]_{\mathbf{q}_{h}}^!\prod_{k\in\mathbb{I}_{0}}[a_{1k}]^!_{\mathbf{q}_{h}}[a_{1k}-t_k]_{\mathbf{q}_{h}}^![t_k]_{\mathbf{q}_{h}}^!}\\
&=\ds\frac{[\frac{a_{00}-1}{2}+t_0]^!_\fc \prod^{m}_{k=1}[a_{0k}+t_{k}+t_{-k}]_{\mathbf{q}_{h}}^!}{[\frac{a_{00}-1}{2}]^!_\fc[t_0]_{\mathbf{q}_{h}}^! \prod^{m}_{k=1}[a_{0k}]_{\mathbf{q}_{h}}^![t_k]_{\mathbf{q}_{h}}^![t_{-k}]_{\mathbf{q}_{h}}^!}.
\end{align*}
For $k\in\mathbb{I}_{1}$, we have $a_{0k}+t_{k}+t_{-k}=a_{1k}=1$. Hence
\begin{align*}
&\quad\ds\frac{\prod^{m+n}_{k=m+1}[a_{0k}+t_{k}+t_{-k}]_{\mathbf{q}_{k}}^! \prod_{k\in\mathbb{I}_{1}}[a_{1k}-t_{k}]^!_{\mathbf{q}_{k}}[a_{1k}]_{\mathbf{\dot{q}}_{h+1}\mathbf{\ddot{q}}_{k}}^!}{\prod^{m+n}_{k=m+1}
[a_{0k}]_{\mathbf{q}_{k}}^!\prod_{k\in\mathbb{I}_{1}}[a_{1k}]^!_{\mathbf{q}_{k}}[a_{1k}-t_k]_{\mathbf{\dot{q}}_{h+1}\mathbf{\ddot{q}}_{k}}^![t_k]_{\mathbf{\dot{q}}_{h+1}\mathbf{\ddot{q}}_{k}}^!}\\
&=\prod^{m+n}_{k=m+1}\ds\frac{[a_{0k}+t_{k}+t_{-k}]_{\mathbf{q}_{k}}^!}{[a_{0k}]_{\mathbf{q}_{k}}^!}
=\prod^{m+n}_{k=m+1}\ds\frac{[a_{0k}+t_{k}+t_{-k}]_{\mathbf{q}_{h}}^!}{[a_{0k}]_{\mathbf{q}_{h}}^![t_{k}]_{\mathbf{q}_{h}}^![t_{-k}]_{\mathbf{q}_{h}}^!}.
\end{align*}

For Part $(2)$, we only show the case $h=0$, while the case $h\in[1..m+n-1]$ is similar to Part $(1)$.
Let $\delta=\delta(C), \mu=\co(C)$. Take $\mathbf{t}=(t_i)_{i\in\mathbb{I}}\,\in\NN^{N}$ with $\sum_{i\in\mathbb{I}}t_{i}=c$ and $t_{i}+t_{-i}\leq a_{0i}$.
Among those $w\in\D_\delta \cap W_\mu$ such that $A^{w}=\check{A}_{\mathbf{t},1}$, there is a unique shortest element $_\mathbf{t} w$ with
\begin{equation*}\label{llw_t}
\ell(_\mathbf{t} w)
=\sum_{j>k}(a_{0j}-t_j)t_k-\sum_{k<j<-k}t_jt_k-\frac{1}{2}\sum_{k<0}t_k(t_k+1),
\end{equation*}
and the lengths of corresponding $_\mathbf{t} w_\ld$ and $_\mathbf{t} w_\nu$ are $\ell(_\mathbf{t} w_\ld)=\sum_{j<k}a_{1j}t_k$, $\ell(_\mathbf{t} w_{\nu})=0$.
Note that $h,h+1\leq m$ since $m>0$, then $\mathbf{q}_{h+1}=q$. We have
\begin{align*}
&\sum_{\tiny\substack{w\in W_{\D_{\delta}\cap W_{\mu}},\\A^w=\check{A}_{\mathbf{t},1}}}\ds\frac{1}{[A]^!_\fc}({\mathbf{\dot{q}}_{h}})^{\ell(w)}\mathrm{[xy]}_{\ld}T_{wg}\mathrm{[xy]}_{\nu}\\
=&\sum_{\tiny\substack{w\in W_{\D_{\delta}\cap W_{\mu}},\\A^w=\check{A}_{\mathbf{t},1}}}\ds\frac{1}{[A]^!_\fc}({\mathbf{\dot{q}}_{h}})^{\ell(w)}
(\mathrm{[xy]}_{\ld}T_{w_{\ld}}) T_{\mathrm{y}^{w}} (T_{w_{\nu}}\mathrm{[xy]}_{\nu})\\
=&\ds\frac{1}{[A]^!_\fc}\frac{[\frac{a_{00}-1}{2}]^!_\fc}{[\frac{a_{00}-1}{2} - t_{0}]^!_\fc [t_{0}]_{\mathbf{q}_{h+1}}^!}\prod_{k=1}^{m+n}\lrb{a_{0k}}{t_k}_{\mathbf{\dot{q}}_{h}\mathbf{\ddot{q}}_{k}}\lrb{a_{0k}-t_k}{t_{-k}}_{\mathbf{\dot{q}}_{h}\mathbf{\ddot{q}}_{k}} {\mathbf{\dot{q}}_{h}}^{\ell(_\mathbf{t} w)}{\mathbf{\ddot{q}}_{h+1}}^{\ell(_\mathbf{t} w_\ld)}e_{\check{A}_{\mathbf{t},1}}(\mathrm{[xy]}_{\nu})\\
=&\ds\frac{[A^{_\mathbf{t} w}]^!_\fc}{[A]^!_\fc}\frac{[\frac{a_{00}-1}{2}]^!_\fc}{[\frac{a_{00}-1}{2} - t_{0}]^!_\fc [t_{0}]_{\mathbf{q}_{h+1}}^!}\prod_{k=1}^{m+n}\lrb{a_{0k}}{t_k}_{\mathbf{\dot{q}}_{h}\mathbf{\ddot{q}}_{k}}\lrb{a_{0k}-t_k}{t_{-k}}_{\mathbf{\dot{q}}_{h}\mathbf{\ddot{q}}_{k}}\\
\nonumber&\quad\quad\cdot{\mathbf{\dot{q}}_{h}}^{\sum_{j>k}(a_{0j}-t_j)t_k-\sum_{k<j<-k}t_jt_k-\frac{1}{2}\sum_{k<0}t_k(t_k+1)}
{\mathbf{\ddot{q}}_{h+1}}^{\sum_{j<k}a_{1j}t_k}e_{\check{A}_{\mathbf{t},1}}(\mathrm{[xy]}_{\nu}).
\end{align*}
Note $\mathbf{\dot{q}}_{h}=1$ for $h=0$. So in order to prove the equation \eqref{e_Cce_A}, we only need to verify the following equation holds:
\begin{equation}\label{eq:12345678}
\ds\frac{[A^{_\mathbf{t} w}]^!_\fc}{[A]^!_\fc}\frac{[\frac{a_{00}-1}{2}]^!_\fc}{[\frac{a_{00}-1}{2} - t_{0}]^!_\fc [t_{0}]_{\mathbf{q}_{h+1}}^!}\prod_{k=1}^{m+n}\lrb{a_{0k}}{t_k}_{\mathbf{\dot{q}}_{h}\mathbf{\ddot{q}}_{k}}\lrb{a_{0k}-t_k}{t_{-k}}_{\mathbf{\dot{q}}_{h}\mathbf{\ddot{q}}_{k}}
=\prod_{k\in\mathbb{I}}\left[\begin{array}{cc}a_{1k}+t_k\\t_k\end{array}\right]_{\mathbf{q}_{h+1}}.
\end{equation}

For the left hand side of the equation \eqref{eq:12345678}, we have
\begin{equation*}\label{eq:12345600}
\textup{LHS}=\ds\frac{\prod^{m+n}_{k=1}[a_{0k}-t_{k}-t_{-k}]_{\mathbf{q}_{k}}^![a_{0k}]_{\mathbf{\dot{q}}_{h}\mathbf{\ddot{q}}_{k}}^! \prod_{k\in\mathbb{I}}[a_{1k}+t_{k}]^!_{\mathbf{q}_{k}}}
{[t_{0}]_{\mathbf{q}_{h+1}}^!\prod^{m+n}_{k=1}[a_{0k}]_{\mathbf{q}_{k}}^![a_{0k}-t_k-t_{-k}]_{\mathbf{\dot{q}}_{h}\mathbf{\ddot{q}}_{k}}^!
[t_k]_{\mathbf{\dot{q}}_{h}\mathbf{\ddot{q}}_{k}}^![t_{-k}]_{\mathbf{\dot{q}}_{h}\mathbf{\ddot{q}}_{k}}^!\prod_{k\in\mathbb{I}}[a_{1k}]^!_{\mathbf{q}_{k}}}.
\end{equation*}
For $k\in\mathbb{I}_{0}$, we have ${\mathbf{q}}_{h+1}={\mathbf{q}}_{k}$ and ${\mathbf{\dot{q}}_{h}\mathbf{\ddot{q}}_{k}}={\mathbf{q}}_{h+1}$, then
\begin{align*}
&\quad\ds\frac{\prod^{m}_{k=1}[a_{0k}-t_{k}-t_{-k}]_{\mathbf{q}_{k}}^![a_{0k}]_{\mathbf{\dot{q}}_{h}\mathbf{\ddot{q}}_{k}}^! \prod_{k\in\mathbb{I}_{0}}[a_{1k}+t_{k}]^!_{\mathbf{q}_{k}}}
{[t_{0}]_{\mathbf{q}_{h+1}}^!\prod^{m}_{k=1}[a_{0k}]_{\mathbf{q}_{k}}^![a_{0k}-t_k-t_{-k}]_{\mathbf{\dot{q}}_{h}\mathbf{\ddot{q}}_{k}}^!
[t_k]_{\mathbf{\dot{q}}_{h}\mathbf{\ddot{q}}_{k}}^![t_{-k}]_{\mathbf{\dot{q}}_{1}\mathbf{\ddot{q}}_{k}}^!\prod_{k\in\mathbb{I}_{0}}[a_{1k}]^!_{\mathbf{q}_{k}}}\\
&=\ds\frac{\prod_{k\in\mathbb{I}_{0}}[a_{1k}+t_{k}]^!_{{\mathbf{q}}_{h+1}}}
{[t_{0}]_{\mathbf{q}_{h+1}}^!\prod^{m}_{k=1}[t_k]_{{\mathbf{q}}_{h+1}}^![t_{-k}]_{{\mathbf{q}}_{h+1}}^!\prod_{k\in\mathbb{I}_{0}}[a_{1k}]^!_{{\mathbf{q}}_{h+1}}}\\
&=\ds\frac{\prod_{k\in\mathbb{I}_{0}}[a_{1k}+t_{k}]^!_{{\mathbf{q}}_{h+1}}}
{\prod_{k\in\mathbb{I}_{0}}[t_k]_{{\mathbf{q}}_{h+1}}^![a_{1k}]^!_{{\mathbf{q}}_{h+1}}}=\prod_{k\in\mathbb{I}_{0}}\left[\begin{array}{cc}a_{1k}+t_k\\t_k\end{array}\right]_{\mathbf{q}_{h+1}}.
\end{align*}
For $k\in\mathbb{I}_{1}$, we assume that $a_{1k}+t_{k}=1$ since $t_{k}$ and $t_{-k}$ cannot be 1 at the same time.
We have $a_{0k}=1$ and $[a_{0k}-t_{k}-t_{-k}]_{\mathbf{q}_{k}}^!=[t_k]_{\mathbf{\dot{q}}_{h+1}\mathbf{\ddot{q}}_{k}}^!=1$. Thus
\begin{align*}
&\quad\ds\frac{\prod\limits^{m+n}_{k=m+1}[a_{0k}-t_{k}-t_{-k}]_{\mathbf{q}_{k}}^![a_{0k}]_{\mathbf{\dot{q}}_{h}\mathbf{\ddot{q}}_{k}}^! \prod\limits_{k\in\mathbb{I}_{1}}[a_{1k}+t_{k}]^!_{\mathbf{q}_{k}}}
{\prod\limits^{m+n}_{k=m+1}[a_{0k}]_{\mathbf{q}_{k}}^![a_{0k}-t_k-t_{-k}]_{\mathbf{\dot{q}}_{h}\mathbf{\ddot{q}}_{k}}^!
[t_k]_{\mathbf{\dot{q}}_{h}\mathbf{\ddot{q}}_{k}}^![t_{-k}]_{\mathbf{\dot{q}}_{h}\mathbf{\ddot{q}}_{k}}^!\prod\limits_{k\in\mathbb{I}_{1}}[a_{1k}]^!_{\mathbf{q}_{k}}}
\\&=1=\prod\limits_{k\in\mathbb{I}_{1}}\left[\begin{array}{cc}a_{1k}+t_k\\t_k\end{array}\right]_{\mathbf{q}_{h+1}}.
\end{align*}
Hence the equation~\eqref{eq:12345678} holds, and so does the formula~\eqref{e_Cce_A} for $h=0$.

The proof is completed.
\end{proof}

\section{Construction of canonical bases}
\subsection{Length formulas}
Denote $I_\fc:=I_\fa\sqcup\{(0,0)\}=(\{0\}\times[0..m+n])\sqcup([1..m+n]\times\mathbb{I})$.
For $A\in\Xi_{m|n, d}$, we set
\begin{align*}
\^\ell(A):=\ell(A)+\ell^{\circ}(A)\qquad\mbox{with}
\end{align*}
\begin{equation*}\label{eq:lAhat}
\ell(A)=\dfrac{1}{2}
\bigg(
\sum_{(i,j) \in I_\fc }
\Bp{
\sum_{\substack{k < i\\ l > j}}
+
\sum_{\substack{k > i\\ l < j}}
}
a^\natural_{ij} a_{kl}
\bigg)\quad\mbox{and}
\quad
\ell^{\circ}(A)=\dfrac{1}{2}
\bigg(
\sum_{(i,j) \in I_\fc }
\Bp{
\sum_{j>l}
+
\sum_{j<l}
}
(-1)^{\hat{i}}a^\natural_{ij} a_{il}
\bigg),
\end{equation*}
where the notation $$a^\natural_{ij}=\left\{\begin{array}{ll}
\frac{a_{00}-1}{2}, &\mbox{if $(i,j)=(0,0)$;}\\
a_{ij}, & \mbox{otherwise}.
\end{array}\right.$$

For $\ld=(\ld^{(0)}|\ld^{(1)}),\mu=(\mu^{(0)}|\mu^{(1)}) \in\Lambda(m|n,d)$ and $g\in\mathcal{D}_{\lambda\mu}^\circ$,
let $g^+_{\ld^{(0)}\mu^{(0)}}$ (resp. $g^+_{\ld^{(1)}\mu^{(1)}}$) be the longest element in the double coset $W_{\ld^{(0)}} g W_{\mu^{(0)}}$ (resp. $W_{\ld^{(1)}} g W_{\mu^{(1)}}$). The notions $g^+_{\ld^{(0)}\mu^{(0)}}$ and $g^+_{\ld^{(1)}\mu^{(1)}}$ make sense because of \eqref{Dsubset}.
Denote by $w_\circ^{\mu^{(0)}} = \id_{\mu^{(0)}\mu^{(0)}}^+$ (resp. $w_\circ^{\mu^{(1)}} = \id_{\mu^{(1)}\mu^{(1)}}^+$) the longest element in the parabolic subgroup $W_{\mu^{(0)}} = W_{\mu^{(0)}} \id W_{\mu^{(0)}}$ (resp. $W_{\mu^{(1)}} = W_{\mu^{(1)}} \id W_{\mu^{(1)}}$).
\begin{prop}
Let $A =\kappa(\ld,g,\mu)\in \Xi_{m|n,d}$. Then we have
\begin{itemize}
\item[(1)] $\ell(A)=\ell(g)$;
  \item[(2)] $\^\ell(A)=\ell(g_{\ld^{(0)}\mu^{(0)}}^{+})-\ell(w_{\circ}^{\mu^{(0)}})-\ell(g_{\ld^{(1)}\mu^{(1)}}^{+})+\ell(w_{\circ}^{\mu^{(1)}})+\ell(g)$.
  \end{itemize}
\end{prop}
\begin{proof}
The item (1) is just a reformulation of Lemma~\ref{lem:length}.

We start to prove the item (2).
 Let $\delta=\delta(A)$ as in \eqref{def:delta} and denote
\begin{align*}
I^{(0)}_\fc:=(\{0\}\times[0..m+n])\sqcup([1..m]\times\mathbb{I}),\quad I^{(1)}_\fc:=I_\fc\setminus I^{(0)}_\fc.
\end{align*}
Thanks to \cite[Corollary 4.19]{DDPW08}, we have
\begin{align*}
&\ell(g_{\ld^{(0)}\mu^{(0)}}^{+})-\ell(w_{\circ}^{\mu^{(0)}})-\ell(g)=\ell(w_{\circ}^{\ld^{(0)}})-\ell(w_{\circ}^{\delta^{(0)}})\\
=&\ld^{2}_{0}+\sum_{i=1}^{m}{\ld_i\choose 2}-
\left(\delta^2_{0}+\sum_{i=1}^{m'}{\delta_{i}\choose 2}\right)\\
=&2\sum_{0\leq j< l}a^{\natural}_{0j}a_{0l}+\sum_{j>0}{a_{0j}+1\choose 2}+\sum_{\substack{j<l,\\1\leq i\leq m}}a_{ij}a_{il}\\
=&\frac{1}{2}\sum_{(i,j)\in I_\fc^{(0)}}(\sum_{j<l}+\sum_{j>l})a_{ij}^\natural a_{il}.
\end{align*}
A similar calculation shows
\begin{align*}
\ell(g_{\ld^{(1)}\mu^{(1)}}^{+})-\ell(w_{\circ}^{\mu^{(1)}})-\ell(g)=\frac{1}{2}\big(\sum_{(i,j)\in I_\fc^{(1)}}(\sum\limits_{j<l}+\sum\limits_{j>l})a_{ij}a_{il}\big).
\end{align*}
Therefore
\begin{align*}
(\ell(g_{\ld^{(0)}\mu^{(0)}}^{+})-\ell(g)-\ell(w_{\circ}^{\mu^{(0)}}))-(\ell(g_{\ld^{(1)}\mu^{(1)}}^{+})-\ell(w_{\circ}^{\mu^{(1)}})-\ell(g))+\ell(g)=\^\ell(A).
\end{align*} as desired in item (2).
\end{proof}

\subsection{Standard basis}
Define $A^{\circ}=(a_{ij}^\circ)_{\mathbb{I}\times \mathbb{I}}$ corresponding to $A\in \Xi_{m|n,d}$
by
\begin{equation*}
a_{ij}^\circ=\left\{\begin{array}{ll}
a_{ij}, &\mbox{if $i\in \mathbb{I}_1$};\\
1, &\mbox{if $(i,j)=(0,0)$};\\
0, & \mbox{otherwise}.
\end{array}\right.
\end{equation*}
We write
\begin{align*}
\widehat{A}:=\ell(A^{\circ})=\dfrac{1}{2}
\sum\limits_{(i,j) \in I_\fc }
\Bp{
\sum\limits_{\substack{k < i\\ l > j}}
+
\sum\limits_{\substack{k > i\\ l < j}}
}
a^{\circ\natural}_{ij} a_{kl}^{\circ}.
\end{align*}
Let
\begin{align*}
\varphi_A:=v^{-\^\ell(A)}e_{A} \quad\mbox{and} \quad
[A]:=(-1)^{\widehat{A}}\varphi_A=(-1)^{\widehat{A}}v^{-\^\ell(A)}e_{A}.
\end{align*}
By Lemma~\ref{lem:e_A}, the following corollary is clear.
\begin{cor}\label{cor:e_A}
The sets $\{\varphi_A~|~A\in\Xi_{m|n,d}\}$ and $\{[A]~|~A\in\Xi_{m|n,d}\}$ form two $\mathcal{A}$-bases of $\mathcal{S}^\jmath_{m|n,d}$.
\end{cor}
We call $\{\varphi_A~|~A\in\Xi_{m|n,d}\}$ and $\{[A]~|~A\in\Xi_{m|n,d}\}$ \emph{standard bases} (see Proposition~\ref{prop:Abar} below).
\begin{rmk}
The coefficient $(-1)^{\widehat{A}}$ in the definition of $[A]$ is inspired by the articles \cite{DGZ20} and \cite{DG14}. It seems useless in the present paper but will be used in our  subsequent work to realize $\imath$quantum supergroups.
\end{rmk}

\subsection{Bar involution} \label{sec:bar}
There is an $\mathcal{A}$-algebra bar involution $\bar{\empty}:\mathcal{H}\rightarrow\mathcal{H}$ defined by setting
$\overline{v} = v^{-1}$ (hence $\overline{q} = q^{-1}$) and $\overline{T_w} = T_{w^{-1}}^{-1}$ for all $w\in W$. In particular, for $s_{i} \in S$, we have
$\overline{T_{s_{i}}} = q^{-1} T_{s_{i}} + (q^{-1} - 1)$.

Following \cite[Theorem 6.3]{DR11}, we define a bar involution $\bar{\empty}: \mathcal{S}^\jmath_{m|n,d}\rightarrow \mathcal{S}^\jmath_{m|n,d}$ as follows:
for any $f\in\Hom_{\mathcal{H}} ([\mathrm{xy}]_{\ld}\mathcal{H}, [\mathrm{xy}]_{\mu}\mathcal{H})$,
let $\overline{f}$ be the $\mathcal{H}$-linear map that sends $[\mathrm{xy}]_{\ld}$ to $\overline{f(\overline{[\mathrm{xy}]_{\ld}})}$.
It is clear that this bar involution preserves the $\ZZ_2$-grading.

For $A\in\Xi_{m|n,d}$, let
$\sigma_{ij}(A) = \sum_{x\le i, y\ge j} a_{xy}$.
Now we define a partial order $\le_\alg$ on $\Xi_{m|n,d}$ by letting, for $A, B \in \Xi_{m|n,d}$,
$$A \le_\alg B \Leftrightarrow
\ro(A) = \ro(B),\; \co(A)=\co(B), \; \text{and } \sigma_{ij}(A) \le \sigma_{ij}(B), \forall i < j.$$
We denote $A <_\alg B$ if $A \le_\alg B$ and $A\neq B$. It is compatible with the Bruhat order on $W$. That is,
\begin{lem}\label{order}
Assume $A=\kappa(\lambda,g,\mu), B=\kappa(\lambda,g',\mu)\in\Xi_{m|n,d}$. If $g<g'$ then $A<_\alg B$.
\end{lem}
\begin{proof}
  Almost the same as the proof of \cite[Corollary~5.5]{FLLLWb}.
\end{proof}

\begin{prop}\label{prop:Abar}
Let $A =\kappa(\ld,g,\mu)\in \Xi_{m|n,d}$. Then we have
\begin{itemize}
  \item[(1)] $\={\varphi_A}  \in \varphi_A +\sum_{B <_\alg A} \mathcal{A} \varphi_B$;
   \item[(2)] $\={[A]}  \in [A] +\sum_{B <_\alg A} \mathcal{A} [B]$.
\end{itemize}
\end{prop}
\begin{proof}
We only prove the statement for $[A]$ here.
We know
\begin{align*}
[A](v^{-\ell(w_{\circ}^{\mu^{(0)}})+\ell(w_{\circ}^{\mu^{(1)}})}[\mathrm{xy}]_{\mu})=
(-1)^{\widehat{A}}v^{-\ell(g_{\ld^{(0)}\mu^{(0)}}^{+})+\ell(g_{\ld^{(1)}\mu^{(1)}}^{+})-\ell(g)}T_{W_\ld g W_\mu}.
\end{align*}
By \cite[Lemma~6.1 \& Proposition~6.2]{DR11}, we deduce that
\begin{align*}
\overline{v^{-\ell(w_{\circ}^{\mu^{(0)}})+\ell(w_{\circ}^{\mu^{(1)}})}[\mathrm{xy}]_{\mu}}
=v^{-\ell(w_{\circ}^{\mu^{(0)}})+\ell(w_{\circ}^{\mu^{(1)}})}[\mathrm{xy}]_{\mu}\qquad \mbox{and}
\end{align*}
\begin{align*}
&\overline{[A]}(v^{-\ell(w_{\circ}^{\mu^{(0)}})+\ell(w_{\circ}^{\mu^{(1)}})}[\mathrm{xy}]_{\mu})\\
=&\overline{[A](\overline{v^{-\ell(w_{\circ}^{\mu^{(0)}})+\ell(w_{\circ}^{\mu^{(1)}})}[\mathrm{xy}]_{\mu}})}
=\overline{[A](v^{-\ell(w_{\circ}^{\mu^{(0)}})+\ell(w_{\circ}^{\mu^{(1)}})}[\mathrm{xy}]_{\mu})}\\
=&\overline{(-1)^{\widehat{A}}v^{-\ell(g_{\ld^{(0)}\mu^{(0)}}^{+})+\ell(g_{\ld^{(1)}\mu^{(1)}}^{+})-\ell(g)}T_{W_\ld g W_\mu}}\\
\in&(-1)^{\widehat{A}}v^{-\ell(g_{\ld^{(0)}\mu^{(0)}}^{+})+\ell(g_{\ld^{(1)}\mu^{(1)}}^{+})-\ell(g)}T_{W_\ld g W_\mu}+\sum_{y<g}\mathcal{A}T_{W_\ld y W_\mu}.
\end{align*}
Note that $[\kappa(\ld,y,\mu)]([\mathrm{xy}]_{\mu}) \in \mathcal{A}T_{W_\ld y W_\mu}$. Induction on $\ell(g)$, we have
\begin{align*}
\overline{[A]}\in [A]+ \sum_{y\in\D^{\circ}_{\ld\mu},~y<g}\mathcal{A}[\kappa(\ld, y, \mu)].
\end{align*}
Hence Lemma~\ref{order} tells us
$\overline{[A]}\in [A] +\sum_{B <_\alg A} \mathcal{A} [B]$.
We are done.
\end{proof}

Recall the definition of Chevalley matrix before Proposition~\ref{thm:e_Bpe_A}.
\begin{cor}\label{cor:bar}
  If $A\in\Xi_{m|n,d}$ is a Chevalley matrix, then both $[A]$ and $\varphi_A$ are bar invariant, i.e. $\overline{[A]}=[A]$ and $\overline{\varphi_A}=\varphi_A$.
\end{cor}
\begin{proof}
  Note that a Chevalley matrix must be in the form of $A=\kappa(\mathrm{row}(A),\id,\mathrm{col}(A))$. The statement follows from Lemma~\ref{order} and Proposition~\ref{prop:Abar}.
\end{proof}

\subsection{Multiplication formulas (revisit)}
\begin{prop}\label{prop:[A]_0}
Let $A, B, C\in\Xi_{m|n,d}$ and $h\in[0..m+n-1]$.
\begin{enumerate}
\item[(1)] Assume $B-E^{\theta}_{h,h+1}$ is diagonal and $\co(B)=\ro(A)$.
For $h\neq0$, we have
\begin{equation*}
[B] [A]=\sum_{k\in\mathbb{I}}(-1)^{\epsilon_{hk}}\mathbf{v}_{h}^{\beta_{hk}}\={[a_{hk}+1]}_{\mathbf{v}^2_{h}} [A_{hk}],
\end{equation*}
where $\beta_{hk}=\sum\limits_{j\geq k}a_{hj}-(-1)^{\widehat{h}+\widehat{h+1}}\sum\limits_{j> k}a_{h+1,j}$ and
$\epsilon_{hk}=\bc{
\quad\quad0 & (h\neq m);\\
\sum\limits_{\tiny \substack{i>m,j<k\,\,\mbox{or}\\i < -m,j > k}}a_{ij} & (h=m).}
$

For $h=0$, we have
\begin{equation*}
[B] [A]=v^{\beta'_{00}}\={[a_{00}+1]}_{\fc} [A_{00}] + \sum_{k\neq 0}(-1)^{\epsilon_{0k}}v^{\beta'_{0k}}\={[a_{0k}+1]}_{v^2} [A_{0k}]
\end{equation*}
where $\beta'_{0k}=\sum\limits_{j\geq k}a_{0j}-(-1)^{\widehat{h}+\widehat{h+1}}\sum\limits_{j> k}a_{1j}$.
\item[(2)] Assume $C-E^{\theta}_{h+1,h}(=C-E^{\theta}_{-h-1,-h})$ is diagonal and $\co(B)=\ro(A)$. For $h\neq0$, we have
\begin{equation*}
[C] [A]=\sum_{k\in\mathbb{I}}(-1)^{\epsilon_{hk}+\delta_{0>k}}\mathbf{v}_{h+1}^{\check{\beta}_{hk}} \={[a_{h+1,k}+1]}_{\mathbf{v}^2_{h+1}} [\check{A}_{hk}],
\end{equation*}
where $\check{\beta}_{hk}=\sum\limits_{j\leq k}a_{h+1,j}-(-1)^{\widehat{h}+\widehat{h+1}}\sum\limits_{j< k}a_{hj}$.

For $h=0$, we have
\begin{equation*}
[C] [A]=\sum_{k\in\mathbb{I}}(-1)^{\epsilon_{0k}+\delta_{0>k}}\mathbf{v}_{h+1}^{\check{\beta}'_{0k}} \={[a_{1k}+1]}_{\mathbf{v}^2_{h+1}} [\check{A}_{0k}],
\end{equation*}
where $\check{\beta}'_{0k}=\sum\limits_{j\leq k}a_{1j}-(-1)^{\widehat{h}+\widehat{h+1}}\sum\limits_{j< k}a_{0j}+\delta_{k>0}(-1)^{\widehat{h}+\widehat{h+1}}$
\\and $\delta_{x>y}=\bc{
1 &\mbox{if}\quad x>y;\\
0 &\mbox{otherwise}.}
$
\end{enumerate}
\end{prop}

\begin{proof}
The statements follow from Proposition~\ref{lem:e_Be_A} and the calculation below:
\begin{equation*}
\widehat{B}+\widehat{A}+\widehat{A}_{hk}=\bc{
2\widehat{A} &\mbox{if}\quad h<m;\\
-\sum\limits_{i>m+1,j<k}a_{ij}-\sum\limits_{i<-m,j>k}a_{ij}+2\widehat{A} &\mbox{if}\quad h=m;\\
-\sum\limits_{j>k}a_{hj}+\sum\limits_{j<k}a_{h+1,j}+2\widehat{A} &\mbox{if}\quad h>m.}
\end{equation*}
\begin{equation*}
\widehat{C}+\widehat{A}+\widehat{\check{A}}_{hk}=\bc{
2\widehat{A} &\mbox{if}\quad h<m;\\
\sum\limits_{i>m+1,j<k}a_{ij}+\sum\limits_{i<-m,j>k}a_{ij}+2\widehat{A}+\delta_{0>k} &\mbox{if}\quad h=m;\\
\sum\limits_{j>k}a_{hj}-\sum\limits_{j<k}a_{h+1,j}+2\widehat{A} &\mbox{if}\quad h>m.}
\end{equation*}
and
\begin{equation*}
\widehat{\ell}(A_{hk})-\widehat{\ell}(A)-\widehat{\ell}(B)=\bc{
-\sum\limits_{j\geq k}a_{hj}-\sum\limits_{j> k}a_{h+1,j} &\mbox{if}\quad h<m;\\
-\sum\limits_{j\geq k}a_{mj}+2\sum\limits_{j<k}a_{m+1,j}+\sum\limits_{j>k}a_{m+1,j} &\mbox{if}\quad h=m;\\
-\sum\limits_{j>k}a_{hj}+2\sum\limits_{j<k}a_{h+1,j}+\sum\limits_{j>k}a_{h+1,j}+a_{hk} &\mbox{if}\quad h>m.}
\end{equation*}
\begin{equation*}
\widehat{\ell}(\check{A}_{hk})-\widehat{\ell}(A)-\widehat{\ell}(C)=\bc{
-\sum\limits_{j\leq k}a_{h+1,j}-\sum\limits_{j<k}a_{hj}+\delta_{h,0}\delta_{k>0} &\mbox{if}\quad h<m;\\
-\sum\limits_{j< k}a_{m+1,j}-\sum\limits_{j< k}a_{mj}+a_{m+1,k}+\delta_{h,0}\delta_{k>0} &\mbox{if}\quad h=m;\\
-\sum\limits_{j< k}a_{h+1,j}-2\sum\limits_{j>k}a_{hj}+\sum\limits_{j<k}a_{hj}+a_{h+1,k} &\mbox{if}\quad h>m.}
\end{equation*}
\end{proof}

\begin{prop}\label{lem:[A]_bc}
Let $A, B, C\in\Xi_{m|n,d}$ and $m\neq h\in[0..m+n-1]$.

\begin{enumerate}
\item[(1)] Assume $B-bE^{\theta}_{h,h+1}$ is diagonal and $\co(B)=\ro(A)$.

For $h\neq0$, we have
\begin{equation*}
[B] [A]=\sum_{\mathbf{t}}\mathbf{v}_{h}^{\beta(\mathbf{t})}  \prod_{k\in\mathbb{I}} \overline{\left[\begin{array}{cc}a_{hk}+t_k\\t_k\end{array}\right]}_{\mathbf{v}^2_{h}} [A_{\mathbf{t},h+1}],
\end{equation*}
where $\mathbf{t}=(t_i)_{i\in\mathbb{I}}\in\mathbb{N}^{N}$ runs over the set as the same as that in \eqref{e_Bbe_A}, and $\beta(\mathbf{t})=\sum\limits_{j\geq k}a_{hj}t_k-\sum\limits_{j> k}(a_{h+1,j}-t_j)t_k$.

For $h=0$, we have
\begin{equation*}
[B] [A]=\sum_{\mathbf{t}}{\mathbf{v}_{h}}^{\beta'(\mathbf{t})}
\overline{\left(\frac{[\frac{a_{00}-1}{2}+t_0]_\fc^!}{[\frac{a_{00}-1}{2}]_\fc^![t_0]^!_{\mathbf{v}^2_{h}}}\prod^{m+n}_{k=1}\frac{[a_{0k}+t_k+t_{-k}]_{\mathbf{v}^2_{h}}^!}
{[a_{0k}]_{\mathbf{v}^2_{h}}^![t_k]_{\mathbf{v}^2_{h}}^![t_{-k}]_{\mathbf{v}^2_{h}}^!}\right)}[A_{\mathbf{t},1}],
\end{equation*}
where $\mathbf{t}=(t_i)_{i\in\mathbb{I}}\in\mathbb{N}^{N}$ runs over the set as the same as that in \eqref{e_B0e_A}, and $\beta'(\mathbf{t})=\sum_{j\geq k}a_{0j}t_k-\sum_{j>k}(a_{1j}-t_j)t_k+\sum_{k<j\leq -k}t_jt_k+\sum_{j\leq 0}\frac{t_j(t_j-1)}{2}$.

\item[(2)] If $C-cE^{\theta}_{h+1,h}(=C-cE^{\theta}_{-h-1,-h})$ is diagonal and $\co(C)=\ro(A)$,
then we have
\begin{equation*}
[C] [A]=\sum_{\mathbf{t}}\mathbf{v}_{h+1}^{\check{\beta}(t)} \prod_{k\in\mathbb{I}} \overline{\left[\begin{array}{cc}a_{h+1,k}+t_k\\t_k\end{array}\right]}_{\mathbf{v}^2_{h+1}} [\check{A}_{\mathbf{t},h+1}],
\end{equation*}
where $\mathbf{t}=(t_i)_{i\in\mathbb{I}}\in\mathbb{N}^{N}$ runs over the set as the same as that in \eqref{e_Cce_A}, and $\check{\beta}(\mathbf{t})=\sum\limits_{j\leq k}a_{h+1,j}t_k-\sum\limits_{j< k}(a_{hj}-t_j)t_k+\delta_{h,0}(\sum\limits_{-k< j< k}t_jt_k+\sum\limits_{j>0}\dfrac{t_j(t_j+1)}{2})$.
\end{enumerate}
\end{prop}
\begin{proof}
Part (1) concludes by combining Proposition~\ref{thm:e_Bpe_A} and the following identities:
\begin{equation*}
\widehat{B}+\widehat{A}+\widehat{A}_{\mathbf{t},h+1}=\bc{
2\widehat{A} &\mbox{if}\quad h<m;\\
-\sum_{j>k}a_{hj}t_k+\sum_{j<k}(a_{h+1,j}-t_j)t_k+2\widehat{A} &\mbox{if}\quad h>m}
\end{equation*}
and
\begin{align*}
\begin{split}
&\^\ell(A_{\mathbf{t},h+1})-\^\ell(A)-\^\ell(B)\\
=&\sum_{j<k}a_{h+1,j}t_{k}-\sum_{j<k}a_{hk}t_{j}-(-1)^{\widehat{h}}\sum_{j}a_{hj}t_{j}-\sum_{j<k}(-1)^{\widehat{h+1}}a_{h+1,k}t_{j}\\
&\quad-\sum_{j<k}(-1)^{\widehat{h+1}}a_{h+1,j}t_{k}+\sum_{j<k}\left((-1)^{\widehat{h}}+(-1)^{\widehat{h+1}}-1\right)t_{k}t_{j}\\
&\quad +\delta_{h,0}(\sum_{-k<j<k}t_j t_k+\sum_{j>0}\frac{t_j(t_j-1)}{2}-\frac{b(b-1)}{2}).
\end{split}
\end{align*}

For part (2),
\begin{equation*}
\widehat{C}+\widehat{A}+\widehat{\check{A}}_{\mathbf{t},h+1}=
\left\{
\begin{array}{ll}
2\widehat{A} &\mbox{if}\quad h<m;\\
-\sum_{j<k}a_{h+1,j}t_k+\sum_{j>k}(a_{hj}-t_j)t_k+2\widehat{A} &\mbox{if}\quad h>m.
\end{array}
\right.
\end{equation*}
and
\begin{align*}
\begin{split}
&\^\ell(\check{A}_{\mathbf{t},h+1})-\^\ell(A)-\^\ell(C)
\\=&-\sum_{j<k}a_{h+1,j}t_{k}+\sum_{j<k}a_{hk}t_{j}-(-1)^{\widehat{h+1}}\sum_{j}a_{h+1,j}t_{j}-\sum_{j<k}(-1)^{\widehat{h}}a_{hk}t_{j}\\
&\quad-\sum_{j<k}(-1)^{\widehat{h}}a_{hj}t_{k}+\sum_{j<k}\left((-1)^{\widehat{h}}+(-1)^{\widehat{h+1}}-1\right)t_{k}t_{j}\\
&\quad+\delta_{h,0}(\sum_{-k< j< k}t_jt_k+\sum_{j>0}\dfrac{t_j(t_j+1)}{2}).
\end{split}
\end{align*}
The proof is completed.
\end{proof}

\begin{rmk}
The multiplication formulas of $\varphi_B\varphi_A$ and $\varphi_C\varphi_A$ can also be obtained by adjusting the coefficients.
\end{rmk}

\subsection{Monomial basis}
\begin{prop}\label{prop:mA}
For any $A\in\Xi_{m|n,d}$, we have
$$\prod_{i<j\in\mathbb{I}}[A(i,j)]=\iota_A[A]+\sum_{B<_\alg A} \mathcal{A} [B],$$
with $\iota_A=1$ or $-1$, and
$[A(i,j)]=
[\mathrm{diag}+a_{ij}E_{j-1,j}^\theta][\mathrm{diag}+a_{ij}E_{j-2,j-1}^\theta]
\cdots[\mathrm{diag}+a_{ij}E_{m-1,m}^\theta]([\mathrm{diag}+E_{m,m+1}^\theta]
\cdots[\mathrm{diag}+E_{-m-1,-m}^\theta])^{a_{ij}}[\mathrm{diag}+a_{ij}E_{-m-2,-m-1}^\theta]\cdots[\mathrm{diag}+a_{ij}E_{i,i+1}^\theta]$ if $m<i<j\leq m+n$;
$[A(i,j)]=[\mathrm{diag}+a_{ij}E_{j-1,j}^\theta][\mathrm{diag}+a_{ij}E_{j-2,j-1}^\theta]
\cdots[\mathrm{diag}+a_{ij}E_{i,i+1}^\theta]$ otherwise, where the diagonal parts are uniquely determined by $\mathrm{row}(A)$ and $\mathrm{col}(A)$.
Here the product $\prod_{i<j\in\mathbb{I}}$ is taken in the lexicographical order: $(i,j)$ proceeds $(i',j')$ if and only if $i>i'$, or $i=i', j>j'$.
\end{prop}
\begin{proof}
  The argument is similar to the proof of \cite[Theorem 3.10]{BKLW18} via the multiplication formulas provided in Propositions~\ref{prop:[A]_0} \& \ref{lem:[A]_bc}. The different point is that we have to take $([\mathrm{diag}+E_{m,m+1}^\theta]
\cdots[\mathrm{diag}+E_{-m-1,-m}^\theta])^{a_{ij}}$ instead of $[\mathrm{diag}+a_{ij}E_{m,m+1}^\theta]\cdots[\mathrm{diag}+a_{ij}E_{-m-1,-m}^\theta]$ to adapt the even-odd trivial intersection property. The coefficient $\iota_A$ appears because of the $(-1)$-powers in the leading term when we product by $[\mathrm{diag}+E_{m,m+1}^\theta]$ and $[\mathrm{diag}+E_{-m-1,-m}^\theta]$ (see Proposition~\ref{prop:[A]_0}).
\end{proof}

There is also a $\varphi_A$ analogue of the above proposition.
To put it in a nutshell, for $A\in\Xi_{m|n,d}$ we have a unique family of Chevalley matrices
$A^{(1)}, \ldots, A^{(x)} \in \Xi_{m|n,d}$ for some $x = x(A) \in \NN$ such that the product is triangular in the following sense:
\begin{equation}\label{eq:mon}
[A^{(1)}]\cdots [A^{(x)}]= \iota_A[A] + \sum_{B<_\alg A} \mathcal{A} [B],
\qquad
\varphi_{A^{(1)}}\cdots \varphi_{A^{(x)}}= \iota'_A\varphi_A + \sum_{B<_\alg A} \mathcal{A} \varphi_B,
\end{equation}

\begin{rmk}
Though we do not need to know whether $\iota_A=1$ or $-1$, we formulate it below:
$\iota_A=(-1)^\square$ with
$$\square=\sum_{\tiny\substack{i>m,\\-m\leq j<0}}a_{ij}(\sum_{\tiny\substack{k<-m,\\j<l<-j}}a_{kl}+\sum_{k<-i}a_{k,-j}+1)
+\sum_{\tiny\substack{i>m,\\j<-m}}a_{ij}+\sum_{\tiny\substack{-m\leq i\leq m,\\j<-m}}\sum_{\tiny\substack{k<-m,\\j<l\leq -j}}a_{ij}a_{kl}.$$
\end{rmk}
Denote
\begin{equation*}
m_A=\iota_A [A^{(1)}]\cdots [A^{(x)}]\qquad\mbox{and}\qquad  \mathfrak{m}_A=\iota'_A\varphi_{A^{(1)}}\cdots \varphi_{A^{(x)}}.
\end{equation*}
\begin{prop}\label{prop:mono}
The sets $\{m_A~|~A\in\Xi_{m|n,d}\}$ and $\{\mathfrak{m}_A~|~A\in\Xi_{m|n,d}\}$ form two $\mathcal{A}$-bases of $\mathcal{S}^\jmath_{m|n,d}$. They satisfy that, for any $A\in\Xi_{m|n,d}$,
\begin{itemize}
  \item[(1)] $m_A=[A]+ \sum_{B<_\alg A} \mathcal{A} [B]$ and $\overline{m_A}=m_A$;
  \item[(2)] $\mathfrak{m}_A=\varphi_A+ \sum_{B<_\alg A} \mathcal{A} \varphi_B$ and $\overline{\mathfrak{m}_A}=\mathfrak{m}_A$.
\end{itemize}
\end{prop}
\begin{proof}
  It follows from \eqref{eq:mon} and Corollary~\ref{cor:bar}.
\end{proof}
We call $\{m_A~|~A\in\Xi_{m|n,d}\}$ (resp. $\{\mathfrak{m}_A~|~A\in\Xi_{m|n,d}\}$) the \emph{monomial basis} relative to the standard basis $\{[A]~|~A\in\Xi_{m|n,d}\}$ (resp. $\{\varphi_A~|~A\in\Xi_{m|n,d}\}$) .

\subsection{Canonical Basis}
By Proposition~\ref{prop:mono} and \cite[24.2.1]{Lu93}, it is now standard to deduce the following theorem.
\begin{thm}\label{thm:canonical}
There exists a unique basis $\{\{A\}|A\in\Xi_{m|n,d}\}$ (resp. $\{\{A\}'|A\in\Xi_{m|n,d}\}$) of $\mathcal{S}^\jmath_{m|n,d}$ satisfying
\begin{align*}
\overline{\{A\}}&=\{A\}\in [A] + \sum_{B <_\alg A} v^{-1}\mathbb{Z}[v^{-1}][B]\\
(\mbox{resp.}\quad \overline{\{A\}'}&=\{A\}'\in \varphi_A + \sum_{B <_\alg A} v^{-1}\mathbb{Z}[v^{-1}]\varphi_B).
\end{align*}
\end{thm}
We call $\{\{A\}|A\in\Xi_{m|n,d}\}$ (resp. $\{\{A\}'|A\in\Xi_{m|n,d}\}$) the \emph{canonical basis} relative to $\{[A]~|~A\in\Xi_{m|n,d}\}$ (resp. $\{\varphi_A~|~A\in\Xi_{m|n,d}\}$).

\section{Isomorphism theorem and semisimplicity criteria}\label{sec:Iso}
\subsection{The field $\mathbb{K}$}
Throughout this section, we take $R=\mathbb{K}$ to be a field of characteristic $\neq2$ containing invertible elements $v$ and $q=v^2\neq0,1$.
Denote
\begin{align*}
f_d(q):=\prod_{i=1-d}^{d-1}(q+q^{-i})\in\mathbb{K}.
\end{align*}
\begin{rmk}
The element $f_d(q)$ is invertible in the field $\mathbb{K}$ if one of the following conditions holds: $\mathrm{(i)}$ $q$ is generic, $\mathrm{(ii)}$ $q$ is an odd root of unity, $\mathrm{(iii)}$ $q$ is a primitive (even) r-$\mathrm{th}$ root of unity for $r>d$.
In particular, if we take $\mathbb{K}=\mathbb{Q}(q)$ then $f_d(q)$ is invertible.
\end{rmk}

We simply denote by $\mathbb{H}$ and $\mathbb{S}^\jmath_{m|n,d}$ the Hecke algebra $\mathcal{H}_{\mathbb{K}}$ and the $\imath$Schur superalgebra $\mathcal{S}^\jmath_{m|n,d;\mathbb{K}}$, respectively.
Moreover, let $\mathbb{H}(\mathfrak{S}_d)$ (resp. $\mathbb{S}_{m|n,d}$) be the Hecke algebra (resp. $q$-Schur superalgebra) of type A over $\mathbb{K}$. We refer to \cite{DR11} for the definition of $q$-Schur superalgebras of type A. It is known that $\mathbb{H}(\mathfrak{S}_d)$ can be regarded as a subalgebra of $\mathbb{H}$ generated by $T_{s_1},T_{s_2},\ldots, T_{s_{d-1}}$.

Base change via $\mathcal{A}\rightarrow \mathbb{K}, v\mapsto v$, we may turn the results, obtained for $\mathcal{S}^\jmath_{m|n,d}$, into the similar ones for $\mathbb{S}^\jmath_{m|n,d}$. For example, Corollary~\ref{cor:iso} becomes
\begin{align*}
\mathbb{S}^\jmath_{m|n,d}=\End_{\mathbb{H}}\big(\mathop{\bigoplus}_{\lambda\in\Lambda(m|n,d)}[\mathrm{xy}]_\lambda\mathbb{H}\big)
\cong\mathrm{End}_{\mathbb{H}}(\mathbb{V}^{\otimes d}_{m|n}),
\end{align*}
where $\mathbb{V}=\mathbb{K}\otimes_{\mathcal{A}}V$ is an $N$-dimensional $\mathbb{K}$-superspace with a basis $\{e_i~|~i\in\mathbb{I}\}$.

\subsection{Morita equivalence of $\mathbb{H}$}
Following \cite[Definition~3.2]{DJ92}, define elements $u_i^\pm\in\mathbb{H}$, for $0\leq i\leq d$, by
\begin{align*}
u_i^+=\prod_{l=0}^{i-1}({T}_{s_l}\cdots {T}_{s_1}{T}_{s_0}{T}_{s_1}\cdots{T}_{s_l}+q^{l}),
\quad u_i^-=\prod_{l=0}^{i-1}({T}_{s_l}\cdots {T}_{s_1}{T}_{s_0}{T}_{s_1}\cdots{T}_{s_l}-q^{l+1}).
\end{align*}
It is understood that $u_0^+=u_0^-=1$. It has been shown in \cite{DJ92} that we do not need to specify the order of the product in the definitions of $u_i^+$ and $u_i^-$.
Since $T_0(T_0-q)=(T_0-q)T_0=-(T_0-q)$, we have
\begin{equation}\label{eq:u-t}
  T_0u_i^-=u_i^-T_0=-u_i \quad \mbox{for all $0\leq i\leq d$}.
\end{equation}

For each pair of non-negative integers $a$, $b$ with $a+b=d$, let
\begin{align*}
w_{a,b} = \left(\ba{{cccccc}1&\ldots&a&a+1&\ldots&a+b\\b+1&\ldots&b+a&1&\ldots&b}\right)\in W \quad\mbox{and}\quad v_{a,b}=u_b^-T_{w_{a,b}}u_a^+\in\mathbb{H},
\end{align*}
where we use the notation $\left(\ba{{cccc}1&2&\ldots&d\\\sigma(1)&\sigma(2)&\ldots&\sigma(d)}\right)$ to identify an element $\sigma\in W$ since $\sigma(0)=0$ and $\sigma(-i)=-\sigma(i)$.
Finally, Dipper and James constructed an idempotent $e_{a,b}\in\mathbb{H}$ under the assumption that $f_d(q)$ is invertible (cf. \cite[Definition~3.27]{DJ92}).

Following are some conclusions obtained in \cite{DJ92}, which we will use later.
\begin{lem}[Dipper-James]\label{lem:u_i^+}
Let $a,b\in\NN$. Then
\begin{enumerate}
\item[(a)] the elements $u_{d}^{\pm}$ belong to the centre of $\mathbb{H}$;
\item[(b)] $u_{b}^{-}\mathbb{H}u_{a}^{+}=0$ if $a+b>d$;
\item[(c)] $e_{a,b}\mathbb{H}=v_{a,b}\mathbb{H}$ for $a+b=d$;
\item[(d)] $e_{a,b}\mathbb{H}e_{a,b}=e_{a,b}(\mathbb{H}(\mathfrak{S}_a)\otimes\mathbb{H}(\mathfrak{S}_b))$ for $a+b=d$;
\item[(e)] there is a Morita equivalence: $\mathbb{H}\simeq \bigoplus_{i=0}^d e_{i,d-i}\mathbb{H}e_{i,d-i}$.
\end{enumerate}
\end{lem}

\subsection{Isomorphism theorem}

Inspired by \cite{LNX20}, we shall show an isomorphism theorem between the $\imath$Schur superalgebras and the (type A) $q$-Schur superalgebras $\mathbb{S}_{m|n,d}$ as follows.
\begin{thm}\label{thm:iso}
  If $f_d(q)$ is invertible in the field $\mathbb{K}$, then there is an isomorphism of $\mathbb{K}$-algebras:
\begin{align*}
\Phi: \mathbb{S}^\jmath_{m|n,d}\ \substack{\sim\\\longrightarrow}\ \bigoplus_{i=0}^d \mathbb{S}_{(m+1)|n,i}\otimes\mathbb{S}_{m|n,(d-i)}.
\end{align*}
\end{thm}
The proof of this theorem is similar to the one in \cite[\S3]{LNX20}. Let us rewrite the arguments in the next subsections to clarify the discrepancy between the classical and super settings.


\subsection{Actions on $q$-tensor superspace}

Recall $\mathbf{e}_{i}$ in \eqref{eq:ei}. For $i\in[0..m+n], k\in[0..d], j\in\mathbb{N}$,  
we define
\begin{equation*}
w^+_{i(k,j)}=\bc{(-1)^{\widehat{i}}(q^{\frac{1+(-1)^{\widehat{i}}}{2}})^j\mathbf{e}_{-i}+q^{k-1}\mathbf{e}_{i}, &\textup{if}\ i\neq 0;\\
(q^{2j+1}+q^{k-1})\mathbf{e}_{i}, &\textup{if}\ i=0,}
\end{equation*}
\begin{equation*}
w^-_{i(k,j)}=\bc{(-1)^{\widehat{i}}(q^{\frac{1+(-1)^{\widehat{i}}}{2}})^j\mathbf{e}_{-i}-q^{k}\mathbf{e}_{i}, &\textup{if}\ i\neq 0;\\
(q^{2j+1}-q^{k})\mathbf{e}_{i}, &\textup{if}\ i=0.}
\end{equation*}
Then for a nondecreasing tuple $\mathbf{i}=(i_{1},i_{2},\cdots,i_{d})\in [0..m+n]^d$,
we define elements $w_\mathbf{i}^+$ and $w_\mathbf{i}^-$ by
\begin{align*}
w_{(i_1)}^+=w^+_{i_1(1,0)},\quad w_{(i_1)}^-=w^-_{i_1(1,0)},
\end{align*}
and inductively (on $d$) as follows:
\begin{align*}
w_\mathbf{i}^+=w_{(i_{1},i_{2},\cdots,i_{d-1})}^+\ast w^+_{i_d(d,j_d)},\quad w_\mathbf{i}^-=w_{(i_{1},i_{2},\cdots,i_{d-1})}^-\ast w^-_{i_d(d,j_d)},
\end{align*}
where $j_d=\mathrm{max}\{l~|~i_{d-l}=i_d\}$ and the product $\ast$ is defined in \eqref{mult:mathbf{e}_i}.
For arbitrary $\mathbf{j}=(j_{1},j_{2},\cdots,j_{d})\in [0..m+n]^d$ with $\mathrm{wt}(\mathbf{j})=\ld$
there is a unique $g\in\mathcal{D}_\ld$ such that $\mathbf{j}_{\ld}=\mathbf{j}g^{-1}$ is nondecreasing. We set
\begin{align*}
w_\mathbf{j}^+=(-1)^{\widehat{\mathbf{j}}}w_{\mathbf{j}_{\ld}}^+T_g \quad \mbox{and} \quad
w_\mathbf{j}^-=(-1)^{\widehat{\mathbf{j}}}w_{\mathbf{j}_{\ld}}^-T_g.
\end{align*}

Below is a super analogue of \cite[Lemma~3.3.1]{LNX20}.

\begin{lem}\label{lem:w_i^+}
\begin{itemize}
\item[(a)]For $\mathbf{i}\in [0..m+n]^d$, we have $\mathbf{e}_{\mathbf{i}}u_d^\pm=w_\mathbf{i}^\pm$.
\item[(b)] If $\mathbf{i}\in [0..m+n]^d$ but $\mathbf{i}\notin [1..m+n]^d$ then
$w_\mathbf{i}^-=\mathbf{e}_{\mathbf{i}}u_d^-=0$.
\end{itemize}
\end{lem}
\begin{proof}
We shall only prove $\mathbf{e}_{\mathbf{i}}u_d^+=w_\mathbf{i}^+$ here, while the argument for $\mathbf{e}_{\mathbf{i}}u_d^-=w_\mathbf{i}^-$ is similar.

If $\mathbf{i}=(0,\ldots,0)$, then by Proposition~\ref{lem:action_H} we have
\begin{align*}
\mathbf{e}_{\mathbf{i}} u_d^+=&(\underbrace{\mathbf{e}_0\ast\cdots\ast \mathbf{e}_0}_d)
\cdot({T}_{s_{d-1}}\cdots T_{s_1}{T}_{s_0}T_{s_1}\cdots{T}_{s_{d-1}}+q^{d-1})u_{d-1}^+\\
=&(\underbrace{\mathbf{e}_0\ast\cdots\ast \mathbf{e}_0}_d)(q^{2d-1}+q^{d-1})u_{d-1}^+\\
=&(\underbrace{\mathbf{e}_0\ast\cdots\ast \mathbf{e}_0}_{d-1})\ast w_{0(d,d-1)}^+
u_{d-1}^+=((\underbrace{\mathbf{e}_0\ast\cdots\ast \mathbf{e}_0}_{d-1})
u_{d-1}^+)\ast w_{0(d,d-1)}^+\\
=&\cdots=w_{0(1,0)}^+\ast\cdots\ast w_{0(d,d-1)}^+=w_{(0,\ldots,0)}^+.
\end{align*}

Next suppose $\mathbf{i}=(i_{1},i_{2},\cdots,i_{d})\neq(0,\ldots,0)$ is nondecreasing and $k\geq0$ is the maximal integer such that $i_{d-k}=\cdots= i_{d-1}= i_d>0$. In this case,
\begin{align*}
  \mathbf{e}_{\mathbf{i}}u_d^+=&(\mathbf{e}_{i_{1}}\ast \cdots \ast \mathbf{e}_{i_{d}})\cdot({T}_{s_{d-1}}\cdots {T}_{s_0}\cdots{T}_{s_{d-1}}+q^{d-1})u_{d-1}^+\\
  =&\mathbf{e}_{i_{1}}\ast \cdots \ast \mathbf{e}_{i_{d-1}}\ast((-1)^{\widehat{i_d}}(q^{\frac{1+(-1)^{\widehat{i_d}}}{2}})^k\mathbf{e}_{-i_{d}}+q^{d-1}\mathbf{e}_{i_{d}})
  u_{d-1}^+\\
  =&\mathbf{e}_{i_{1}}\ast \cdots \ast \mathbf{e}_{i_{d-1}}\ast w^+_{i_d(d,k)}u_{d-1}^+\\
  =&\left((\mathbf{e}_{i_{1}}\ast \cdots \ast \mathbf{e}_{i_{d-1}})u_{d-1}^+\right)\ast w^+_{i_d(d,k)},
\end{align*}
which suffices to obtain $\mathbf{e}_{\mathbf{i}}u_d^+=w_{\mathbf{i}}^+$ for nondecreasing $\mathbf{i}$ by induction on $d$.

For general $\mathbf{i}$, we can find a unique $g\in\mathcal{D}_{\mathrm{wt}(\mathbf{i})}$ such that $\mathbf{i}_{\mathrm{wt}(\mathbf{i})}=\mathbf{i}g^{-1}$ is nondecreasing. Since $u_d^\pm$ commute with the elements of $\mathbb{H}$ by lemma~\ref{lem:u_i^+}, then we have
\begin{align}\label{def:v_{i}}
\mathbf{e}_{\mathbf{i}}u_d^\pm=(-1)^{\widehat{\mathbf{i}}}\mathbf{e}_{\mathbf{i}_{\mathrm{wt}(\mathbf{i})}}T_gu_d^\pm
=(-1)^{\widehat{\mathbf{i}}}\mathbf{e}_{\mathbf{i}_{\mathrm{wt}(\mathbf{i})}}u_d^\pm T_g=(-1)^{\widehat{\mathbf{i}}}w_{\mathbf{i}_{\mathrm{wt}(\mathbf{i})}}^\pm T_g=w_{\mathbf{i}}^\pm,
\end{align} where the first equality follows from Lemma~\ref{eq:eieg}.

Now let $\mathbf{i}\in [0..m+n]^d$ but $\mathbf{i}\notin [1..m+n]^d$. We shall prove $\mathbf{e}_{\mathbf{i}}u_d^-=0$. Without loss of generality, we may assume $\mathbf{i}$ is nondecreasing in the spirit of \eqref{def:v_{i}}. So $\mathbf{e_i}=\mathbf{e}_0\ast(\cdots)$ and hence $\mathbf{e_i}(T_0-q)=0$. Therefore $\mathbf{e}_{\mathbf{i}}u_d^-=0$.

The lemma is proved.
\end{proof}

\subsection{Decompositions}
Consider the decompositions of $\mathbb{V}_{m|n}$ into $\mathbb{K}$-subspaces:
\begin{align*}
\mathbb{V}_{m|n}=\mathbb{V}_{m|n,\geq 0}\oplus \mathbb{V}_{m|n, <0}=\mathbb{V}_{m|n, >0}\oplus \mathbb{V}_{m|n,\leq 0},
\end{align*}
where
\begin{eqnarray}
&\label{def:space1}~~\mathbb{V}_{m|n,\geq 0}=\bigoplus_{0\leq i\leq m+n}\mathbb{K}e_{i},\quad \mathbb{V}_{m|n, <0}=\bigoplus_{-n-m\leq i\leq -1}\mathbb{K}e_{i},\\
&\label{def:space2}\mathbb{V}_{m|n, >0}=\bigoplus_{1\leq i\leq m+n}\mathbb{K}e_{i},\quad \mathbb{V}_{m|n, \leq0}=\bigoplus_{-n-m\leq i\leq 0}\mathbb{K}e_{i}.
\end{eqnarray}
Hence by \cite[Corollary~8.4]{DR11} with base change via $\mathcal{A}\rightarrow \mathbb{K}$, we have the following canonical isomorphisms:
\begin{align*}
\mathbb{S}_{(m+1)|n,d}=\mathrm{End}_{\mathbb{H}(\mathfrak{S}_d)}(\mathbb{V}^{\otimes d}_{m|n,\geq 0}),\quad\mathbb{S}_{m|n,d}=\mathrm{End}_{\mathbb{H}(\mathfrak{S}_d)}(\mathbb{V}^{\otimes d}_{m|n,< 0}).
\end{align*}


\begin{lem}\label{lem:W^d}
We have $\mathbb{V}_{m|n}^{\otimes d}u_d^{-}=\mathbb{V}_{m|n,>0}^{\otimes d}u_d^{-}$ and $\mathbb{V}_{m|n}^{\otimes d}u_d^{+}=\mathbb{V}_{m|n,\geq 0}^{\otimes d}u_d^{+}$.
\end{lem}
\begin{proof}
We only show the first equation, while the second one is similar.
For any $1\leq i\leq d$, we have
\begin{align*}
&(\mathbb{V}_{m|n,>0}^{\otimes(i-1)}\otimes \mathbb{V}_{m|n,\leq0}\otimes \mathbb{V}_{m|n}^{\otimes(d-i)})u_d^{-}\\
=&(\mathbb{V}_{m|n,>0}^{\otimes(i-1)}\otimes \mathbb{V}_{m|n,<0}\otimes \mathbb{V}_{m|n}^{\otimes(d-i)})u_d^{-}\qquad\mbox{by Lemma~\ref{lem:w_i^+}~(b)}\\
=&(\mathbb{V}_{m|n,>0} \otimes\mathbb{V}_{m|n,>0}^{\otimes(i-1)} \otimes \mathbb{V}_{m|n}^{\otimes(d-i)})T_{0}T_{1}\cdots T_{i-1}u_d^{-}\\
=&(\mathbb{V}_{m|n,>0} \otimes\mathbb{V}_{m|n,>0}^{\otimes(i-1)} \otimes \mathbb{V}_{m|n}^{\otimes(d-i)})u_d^{-}T_{0}T_{1}\cdots T_{i-1}\qquad\mbox{by Lemma~\ref{lem:u_i^+}~(a)}\\
=&(\mathbb{V}_{m|n,>0} \otimes\mathbb{V}_{m|n,>0}^{\otimes(i-1)} \otimes \mathbb{V}_{m|n}^{\otimes(d-i)})(-1)T_{1}\cdots T_{i-1}u_d^{-}\qquad\mbox{by \eqref{eq:u-t}}\\
\subseteq &(\mathbb{V}_{m|n,>0}^{\otimes i} \otimes \mathbb{V}_{m|n}^{\otimes(d-i)})u_d^{-},
\end{align*}
which tells us $(\mathbb{V}_{m|n}^{\otimes i}\otimes \mathbb{V}^{\otimes(d-i)}_{m|n})u_d^{-}=(\mathbb{V}_{m|n,>0}^{\otimes i}\otimes \mathbb{V}^{\otimes(d-i)}_{m|n})u_d^{-}$ by a recursion.
So the result is derived by taking $i=d$.
\end{proof}

\begin{lem}\label{lem:V^d}
We have $V_{m|n}^{\otimes d}v_{a,b}=(V_{m|n,>0}^{\otimes b}\otimes V_{m|n,\geq0}^{\otimes a})v_{a,b}$ for $a+b=d$.
\end{lem}
\begin{proof}
First by Lemma~\ref{lem:W^d} we have
\begin{align*}
V_{m|n}^{\otimes d}v_{a,b}=(V_{m|n}^{\otimes b}\otimes V_{m|n}^{\otimes a})u_b^-T_{w_{a,b}}u_a^+=(V_{m|n,>0}^{\otimes b}\otimes V_{m|n}^{\otimes a})u_b^-T_{w_{a,b}}u_a^+=(V_{m|n,>0}^{\otimes b}\otimes V_{m|n}^{\otimes a})v_{a,b}.
\end{align*}
Then similar to the discussion of Lemma~\ref{lem:W^d}, for $b<i\leq d$, we have
\begin{align*}
\begin{array}{llll}
\quad\,\,(\mathbb{V}^{\otimes b}_{m|n,>0}\otimes\mathbb{V}_{m|n,\geq0}^{(i-b-1)}\otimes \mathbb{V}_{m|n,<0}\otimes \mathbb{V}_{m|n}^{(d-i)})v_{a,b} &~\\
=(\mathbb{V}_{m|n,>0}\otimes\mathbb{V}^{\otimes b}_{m|n,>0}\otimes\mathbb{V}_{m|n,\geq0}^{(i-b-1)}\otimes \mathbb{V}_{m|n}^{(d-i)})T_{0}T_{1}\cdots T_{i-1}v_{a,b} &~\\
=q^{b+1}(\mathbb{V}_{m|n,>0}\otimes\mathbb{V}^{\otimes b}_{m|n,>0}\otimes\mathbb{V}_{m|n,\geq0}^{(i-b-1)}\otimes \mathbb{V}_{m|n}^{(d-i)})T^{-1}_1\cdots T^{-1}_{b}(T_{b+1}\cdots T_{i-1})v_{a,b} \\
\subseteq(\mathbb{V}^{\otimes b}_{m|n,>0}\otimes\mathbb{V}_{m|n,\geq0}^{(i-b)}\otimes \mathbb{V}_{m|n}^{(d-i)})v_{a,b}, &~
\end{array}
\end{align*}
where the second equation is due to
\begin{align*}
&\quad\,\, T_0T_1\cdots T_{i-1}v_{a,b}&\\
&=T^{-1}_1\cdots T^{-1}_{b}(T_{b}\cdots T_{1}T_{0}T_{1}\cdots T_{b})(T_{b+1}\cdots T_{i-1})u_b^-T_{w_{a,b}}u_a^+\\
&=T^{-1}_1\cdots T^{-1}_{b}(T_{b}\cdots T_{1}T_{0}T_{1}\cdots T_{b})u_b^-(T_{b+1}\cdots T_{i-1})T_{w_{a,b}}u_a^+\quad\textup{by Lemma}~\ref{lem:u_i^+}~(a)\\
&=T^{-1}_1\cdots T^{-1}_{b}(u_{b+1}^-+q^{b+1}u_b^-)(T_{b+1}\cdots T_{i-1})T_{w_{a,b}}u_a^+\\
&=q^{b+1}T^{-1}_1\cdots T^{-1}_{b}u_b^-(T_{b+1}\cdots T_{i-1})T_{w_{a,b}}u_a^+\quad\quad\quad\quad\quad\quad\quad\quad\textup{by Lemma}~\ref{lem:u_i^+}~(b)\\
&=q^{b+1}T^{-1}_1\cdots T^{-1}_{b}(T_{b+1}\cdots T_{i-1})u_b^-T_{w_{a,b}}u_a^+\quad\quad\quad\quad\quad\quad\quad\quad\textup{by Lemma}~\ref{lem:u_i^+}~(a)\\
&=q^{b+1}T^{-1}_1\cdots T^{-1}_{b}(T_{b+1}\cdots T_{i-1})v_{a,b}.
\end{align*}
Then we get $(V_{m|n}^{\otimes b}\otimes V_{m|n}^{\otimes (i-b)}\otimes V_{m|n}^{\otimes (d-i)})v_{a,b}=(V_{m|n,>0}^{\otimes b}\otimes V_{m|n,\geq0}^{\otimes (i-b)}\otimes V_{m|n}^{\otimes (d-i)})v_{a,b}$. The lemma is proved by taking $i=d$.
\end{proof}
For $a,b\in\mathbb{N}$ with $a+b=d$, we define some projections:
\begin{align*}
p_d: V_{m|n}^{\otimes d} \rightarrow V_{m|n,\leq0}^{\otimes d};\quad
p_{a,b}: V_{m|n}^{\otimes d} \rightarrow V_{m|n,\leq0}^{\otimes a}\otimes V_{m|n,<0}^{\otimes b};\quad
p'_{a,b}: V_{m|n}^{\otimes d} \rightarrow V_{m|n}^{\otimes a}\otimes V_{m|n,<0}^{\otimes b}.
\end{align*}

Then we have the following lemma:
\begin{lem}\label{lem:p_d}
For any $\mathbf{i}\in [0..m+n]^d$, we have $p_d(w_\mathbf{i}^+)=c'_{\mathbf{i}}\mathbf{e}_{-\mathbf{i}}=c_{\mathbf{i}}e_{-\mathbf{i}}$ for some $c'_{\mathbf{i}}, c_{\mathbf{i}}\in\mathbb{K}^\times$.
Moreover, if $\mathbf{i}\in[0..m+n]^d$ then $p_d(w_\mathbf{i}^-)=p_d(w_\mathbf{i}^+)=c'_{\mathbf{i}}\mathbf{e}_{-\mathbf{i}}=c_{\mathbf{i}}e_{-\mathbf{i}}$.
\end{lem}
\begin{proof}
The results follow from the definitions of $w_\mathbf{i}^\pm$ and $\mathbf{e}_{\mathbf{i}}=v^{|\mathbf{i}|}e_{\mathbf{i}}$ directly.
\end{proof}

\begin{lem}\label{lem:p_a,b}
Let $a+b=d$. For $\mathbf{i}\in[0..m+n]^a$ and $\mathbf{j}\in[1..m+n]^b$, we have
\begin{align*}
p_{a,b}((e_{\mathbf{j}}\otimes e_{\mathbf{i}})v_{a,b})=c^{\mathbf{i},\mathbf{j}}e_{\mathbf{-i}}\otimes e_{\mathbf{-j}}\quad\mbox{for some}\quad c^{\mathbf{i},\mathbf{j}}\in\mathbb{K}^{\times}.
\end{align*}
\end{lem}
\begin{proof}

Note that $e_{\mathbf{j}}\otimes e_{\mathbf{i}}=v^{-|(\mathbf{j},\mathbf{i})|}(\mathbf{e}_{\mathbf{j}}\ast \mathbf{e}_{\mathbf{i}})$ by \eqref{mult:mathbf{e}_i}. Proposition~\ref{lem:action_H} implies
$$(w_\mathbf{j}^-\otimes e_{\mathbf{i}})T_{w_{a,b}}=c_{\mathbf{i},\mathbf{j}}(w_\mathbf{j}^-\otimes e_{\mathbf{i}})w_{a,b}
+\sum_{g<w_{a,b}}c_g(w_\mathbf{j}^-\otimes e_{\mathbf{i}})g$$
for some element $c_{\mathbf{i},\mathbf{j}}\in\mathbb{K}^{\times}$ and $c_{g}\in\mathbb{K}$. Then we have
\begin{align}\label{claim4.8}
p'_{a,b}((w_\mathbf{j}^-\otimes e_{\mathbf{i}})T_{w_{a,b}})&=p'_{a,b}(c_{\mathbf{i},\mathbf{j}}(w_\mathbf{j}^-\otimes e_{\mathbf{i}})w_{a,b}+\sum_{g<w_{a,b}}c_g(w_\mathbf{j}^-\otimes e_{\mathbf{i}})g)\\
\nonumber&=c_{\mathbf{i},\mathbf{j}}p'_{a,b}(e_{\mathbf{i}}\otimes w_\mathbf{j}^-)
+\sum_{g<w_{a,b}}c_{g}p'_{a,b}((w_\mathbf{j}^-\otimes e_{\mathbf{i}})g)\\
\nonumber&=c_{\mathbf{i},\mathbf{j}}p'_{a,b}(e_{\mathbf{i}}\otimes w_\mathbf{j}^-)=c_{\mathbf{i},\mathbf{j}}(e_{\mathbf{i}}\otimes p_b(w_\mathbf{j}^-))=c_{\mathbf{i},\mathbf{j}}c_{\mathbf{j}}(e_{\mathbf{i}}\otimes e_{\mathbf{-j}}).
\end{align}
Therefore
\begin{align*}
&\quad\,\, p_{a,b}((e_{\mathbf{j}}\otimes e_{\mathbf{i}})v_{a,b})\\
&=p_{a,b}((e_{\mathbf{j}}\otimes e_{\mathbf{i}}) u_b^-T_{w_{a,b}}u_a^+)\overset{\textup{By \eqref{def:mathbf{e}_i}}}{=}v^{-|\mathbf{j}|}p_{a,b}((\mathbf{e}_{\mathbf{j}}\otimes e_{\mathbf{i}}) u_b^-T_{w_{a,b}}u_a^+)\\
&\overset{\textup{Lem.}~\ref{lem:w_i^+}~(b)}{=}v^{-|\mathbf{j}|}p_{a,b}((w^-_{\mathbf{j}}\otimes e_{\mathbf{i}}) T_{w_{a,b}}u_a^+)=v^{-|\mathbf{j}|}p_{a,b}(p'_{a,b}((w^-_{\mathbf{j}}\otimes e_{\mathbf{i}}) T_{w_{a,b}})u_a^+)\\
&\overset{\textup{By \eqref{claim4.8}}}{=}v^{-|\mathbf{j}|}p_{a,b}(c_{\mathbf{i},\mathbf{j}}c_{\mathbf{j}}(e_{\mathbf{i}}\otimes e_{\mathbf{-j}})u_a^+)
\overset{\textup{By \eqref{def:mathbf{e}_i}}}{=}v^{-|\mathbf{j}|-|\mathbf{i}|}p_{a,b}(c_{\mathbf{i},\mathbf{j}}c_{\mathbf{j}}(\mathbf{e}_{\mathbf{i}}\otimes e_{\mathbf{-j}})u_a^+)\\
&\overset{\textup{Lem.}~\ref{lem:w_i^+}~(a)}{=}v^{-|\mathbf{j}|-|\mathbf{i}|}p_{a,b}(c_{\mathbf{i},\mathbf{j}}c_{\mathbf{j}}(w^+_{\mathbf{i}}\otimes e_{\mathbf{-j}}))
=v^{-|\mathbf{j}|-|\mathbf{i}|}c_{\mathbf{i},\mathbf{j}}c_{\mathbf{j}}p_a(w^+_{\mathbf{i}})\otimes e_{\mathbf{-j}}\\
&\overset{\textup{Lem.}~\ref{lem:p_d}}{=}v^{-|\mathbf{j}|-|\mathbf{i}|}c_{\mathbf{i},\mathbf{j}}c_{\mathbf{j}}c_{\mathbf{i}}(e_{\mathbf{-i}}\otimes e_{\mathbf{-j}}).
\end{align*}
Let $c^{\mathbf{i},\mathbf{j}}=v^{-|\mathbf{j}|-|\mathbf{i}|}c_{\mathbf{i},\mathbf{j}}c_{\mathbf{j}}c_{\mathbf{i}}$, then we have done.
\end{proof}

\begin{lem}\label{lem:V_iso}
For $a,b\in\NN$ with $a+b=d$, we have an $\mathbb{H}(\mathfrak{S}_a)\otimes\mathbb{H}(\mathfrak{S}_b)$-module isomorphism
\begin{align*}
\varphi: \mathbb{V}^{\otimes a}_{m|n, \geq 0}\otimes \mathbb{V}^{\otimes b}_{m|n, > 0}\rightarrow \mathbb{V}^{\otimes d}_{m|n}v_{a,b},\quad e_\mathbf{i}\otimes e_\mathbf{j}\mapsto(e_\mathbf{j}\otimes e_\mathbf{i})v_{a,b},
\end{align*}
where $\mathbf{i}\in[0..m+n]^a$ and $\mathbf{j}\in[1..m+n]^b$.
\end{lem}
\begin{proof}
According to \cite[Lemma 3.10]{DJ92}, we know
$
T_iv_{a,b}=\bc{v_{a,b}T_{i+a}, &\textup{if}\ 1\leq i <b;\\
v_{a,b}T_{i-b}, &\textup{if}\ b< i < a+b.}
$
Thus the map $\varphi$ is an $\mathbb{H}(\mathfrak{S}_a)\otimes\mathbb{H}(\mathfrak{S}_b)$-module homomorphism.

Let $e_\mathbf{i}\otimes e_\mathbf{j}\neq e_\mathbf{i'}\otimes e_\mathbf{j'} \in\mathbb{V}^{\otimes a}_{m|n, \geq 0}\otimes \mathbb{V}^{\otimes b}_{m|n, > 0}$. Suppose $(e_\mathbf{j}\otimes e_\mathbf{i})v_{a,b}= (e_\mathbf{j'}\otimes e_\mathbf{i'})v_{a,b}$,
then by Lemma~\ref{lem:p_a,b} we have
\begin{align*}
c^{\mathbf{i},\mathbf{j}}e_{\mathbf{-i}}\otimes e_{\mathbf{-j}}=p_{a,b}((e_\mathbf{j}\otimes e_\mathbf{i})v_{a,b})
=p_{a,b}((e_\mathbf{j'}\otimes e_\mathbf{i'})v_{a,b})=c^{\mathbf{i'},\mathbf{j'}}e_{\mathbf{-i'}}\otimes e_{\mathbf{-j'}},
\end{align*}
which is impossible since $c^{\mathbf{i},\mathbf{j}}$ and $c^{\mathbf{i'},\mathbf{j'}}$ are both nonzero. So $\varphi$ is injective. The surjectivity is obvious by Lemma~\ref{lem:V^d}. The proof is completed.
\end{proof}

\subsection{Proof of the isomorphism theorem}
Now it is the time to verify Theorem~\ref{thm:iso}.
\begin{proof} We can compute that
\begin{align*}
\begin{array}{lllllll}
\mathbb{S}^\jmath_{m|n,d} &=\mathrm{End}_{\mathbb{H}}(\mathbb{V}^{\otimes d}_{m|n}) &~\\
~ &=\mathrm{End}_{\bigoplus\limits_{0\leq i\leq d}e_{i,d-i}\mathbb{H}e_{i,d-i}}(\bigoplus\limits_{0\leq i\leq d}\mathbb{V}^{\otimes d}_{m|n}e_{i,d-i}) &\textup{by Lemma}~\ref{lem:u_i^+}~(e)\\
~ &=\bigoplus\limits_{0\leq i\leq d}\mathrm{End}_{e_{i,d-i}\mathbb{H}e_{i,d-i}}(\mathbb{V}^{\otimes d}_{m|n}e_{i,d-i}) &~\\
~ &=\bigoplus\limits_{0\leq i\leq d}\mathrm{End}_{\mathbb{H}(\mathfrak{S}_i)\otimes\mathbb{H}(\mathfrak{S}_{d-i})}(\mathbb{V}^{\otimes d}_{m|n}v_{i,d-i}). &\textup{by Lemma}~\ref{lem:u_i^+}~(c),(d)\\
~ &=\bigoplus\limits_{0\leq i\leq d}\mathrm{End}_{\mathbb{H}(\mathfrak{S}_i)\otimes\mathbb{H}(\mathfrak{S}_{d-i})}(\mathbb{V}^{\otimes i}_{m|n, \geq 0}\otimes \mathbb{V}^{\otimes (d-i)}_{m|n, > 0}) &\textup{by Lemma}~\ref{lem:V_iso}\\
~ &=\bigoplus\limits_{0\leq i\leq d}\mathrm{End}_{\mathbb{H}(\mathfrak{S}_i)}(\mathbb{V}^{\otimes i}_{m|n,\geq 0})\otimes\mathrm{End}_{\mathbb{H}(\mathfrak{S}_{d-i})}(\mathbb{V}^{\otimes (d-i)}_{m|n,> 0}) &~\\
~ &=\bigoplus\limits_{0\leq i\leq d}\mathbb{S}_{(m+1)|n,i}\otimes\mathbb{S}_{m|n,(d-i)} &~
\end{array}
\end{align*} as desired.
\end{proof}

\subsection{Semisimplicity criteria}\label{semcri}

Semisimplicity criteria of the $q$-Schur superalgebra $\mathbb{S}_{m|n,d}$ have been given in \cite{DGZ20}.
\begin{thm}[Du-Gu-Zhou]\label{thm:semi_A}
The $q$-Schur superalgebra $\mathbb{S}_{m|n,d}~(m,n\geq 1)$ is semisimple if and only if one of the following holds:
\begin{enumerate}
\item[(1)] $q$ is not a root of unity;

\item[(2)] $q$ is a primitive $r$-th root of unity with $r>d$;

\item[(3)] $m=n=1$ and $q$ is a primitive $r$-th root of unity with $r \nmid d$.
\end{enumerate}
\end{thm}

We provide semisimplicity criteria of $\mathbb{S}^\jmath_{m|n,d}$ in the following theorem.

\begin{thm}\label{thm:semi_B}
Suppose $f_d(q)$ is invertible in the field $\mathbb{K}$. The $\imath$Schur superalgebra $\mathbb{S}_{m|n,d}^\jmath$ (over $\mathbb{K}$) with $m,~n\geq 0$ is semisimple if and only if one of the following holds:
\begin{enumerate}
\item[(1)] $q$ is not a root of unity;

\item[(2)] $q$ is a primitive r-th root of unity with $r>d$;

\item[(3)] $(m,n)=(0,0)~~or~~(1,0)$.
\end{enumerate}
\end{thm}
\begin{proof}
When $n=0$, This is the non-super case, we can find the results in \cite[Theorem~7.5.1]{LNX20}.

When $n>0$, the semisimplicity of cases (1) and (2) follows from the Theorem~\ref{thm:iso} and Theorem~\ref{thm:semi_A}.
We now show that, if both of these two conditions fail, then $\mathbb{S}_{m|n,d}^\jmath$ is not semisimple.
Assume that $q$ is a primitive $r$-$\mathrm{th}$ root of unity with $r\leq d$,
We first consider the case $(m,n)=(0,1)$, for which we have
\begin{align*}
\mathbb{S}^\jmath_{0|1,d}\cong\bigoplus_{i=0}^d \mathbb{S}_{1|1,i}\otimes\mathbb{S}_{0|1,(d-i)}.
\end{align*}
If $r\leq d$, we can always take $i=r$, then $r|i$, so $\mathcal{S}_{1|1,r}$ is not semisimple by Theorem~\ref{thm:semi_A} $(3)$.
It follows that $\mathbb{S}^\jmath_{0|1,d}$ is not semisimple since $\mathbb{S}_{1|1,r}$ appears as a summand.
The case of $(m,n)=(1,1)$ is similar.
Next we consider $m,~n>1$, for which we have
\begin{align*}
\mathbb{S}^\jmath_{m|n,d}\cong\bigoplus_{i=0}^d \mathbb{S}_{m+1|n,i}\otimes\mathbb{S}_{m|n,(d-i)}.
\end{align*}
Note that $(m+1,n),(m,n)\neq (1,1)$, then $\mathbb{S}^\jmath_{m|n,d}$ is not semisimple by Theorem~\ref{thm:semi_A}.
\end{proof}

\section{A variant of $\imath$Schur superalgebras}
In this section, we shall introduce a variant of $\imath$Schur superalgebras. Due to length limitation, we only list the statements without proofs here since they can be proved by similar arguments for $\mathcal{S}^\jmath_{m|n,d}$.

\subsection{The $\imath$Schur superalgebra $\mathcal{S}^\imath_{m|n,d}$}
Let
\begin{equation*}
\Lambda^\imath(m|n,d):=\left\{\lambda\in\Lambda(m|n,d)~\middle|~\lambda_0=1\right\}\qquad \mbox{and}
\end{equation*}
\begin{align*}
\Xi^\imath_{m|n,d}:=\{A\in\Xi_{m|n,d}~|~\mbox{$\ro(A)_0=1=\co(A)_0$}\}.
\end{align*}
Then the following lemma is an $\imath$-analogue of Lemma~\ref{lem:kappacirc}.
\begin{lem}\label{lem:i_kappacirc}
  The map
  $$\kappa^\imath: \bigsqcup_{\ld,\mu\in\Lambda^\imath(m|n,d)} \{\ld\} \times \D_{\ld\mu}^\circ \times \{\mu\} \to \Xi_{m|n,d}^\imath,
\quad
\kappa^\imath(\ld,g,\mu) = (|R_i^\ld \cap g R_j^\mu|)_{i,j\in\mathbb{I}}$$ is a bijection.
\end{lem}

The $\imath$Schur superalgebra $\mathcal{S}^\imath_{m|n,d}$ is defined as the subsuperalgebra of $\mathcal{S}^\jmath_{m|n,d}$ as follows:
$$\mathcal{S}^\imath_{m|n,d}
:=\End_{\mathcal{H}}\big(\bigoplus_{\lambda\in\Lambda^\imath(m|n,d)}[\mathrm{xy}]_\lambda\mathcal{H}\big)\subset\mathcal{S}^\jmath_{m|n,d}.$$
%
\begin{lem}\label{lem:i_e_A}
The sets $\{e_A~|~A\in\Xi^\imath_{m|n,d}\}$ and $\{[A]~|~A\in\Xi^\imath_{m|n,d}\}$ are both $\mathcal{A}$-bases of $\mathcal{S}^\imath_{m|n,d}$. Hence
$\mathrm{rank}(\mathcal{S}^\imath_{m|n,d})=\sum_{k=0}^{d}{2m^2+2n^2+k-1\choose k}{4mn\choose d-k}$.
\end{lem}

\subsection{Monomial and canonical bases for $\mathcal{S}^\imath_{m|n,d}$}
Fix any $A\in\Xi^\imath_{m|n,d}$. Let us have a careful look at $m_A=\iota_A\prod_{i<j\in\mathbb{I}}[A(i,j)]\in\mathcal{S}^\jmath_{m|n,d}$ by Proposition~\ref{prop:mA}. Its factors $[\mathrm{diag}+aE_{01}^\theta]$ and $[\mathrm{diag}+aE_{10}^\theta]$ $(a\neq0)$ do not lie in $\mathcal{S}^\imath_{m|n,d}$. Note that for $A\in\Xi^\imath_{m|n,d}$, the first entry of $\mathrm{row}([\mathrm{diag}+aE_{10}^\theta])$ is always $0$ and hence the $(0,0)$-entry of the matrix $\mathrm{diag}+aE_{10}^\theta$ is always $1$. So the multiplication formulas in Proposition~\ref{lem:[A]_bc} imply
$$[\mathrm{diag}+aE_{01}^\theta][\mathrm{diag}+aE_{10}^\theta]
\in\sum_{b=0}^{a}\mathcal{A}[\mathrm{diag}+bE_{n,n+2}^\theta]\subset\mathcal{S}^\imath_{m|n,d}.$$
It can be observed from Proposition~\ref{prop:mA} directly that the monomial basis element $m_A$ is always a product of the above ``twin product'' $[\mathrm{diag}+aE_{01}^\theta][\mathrm{diag}+aE_{10}^\theta]\in\mathcal{S}^\imath_{m|n,d}$ together with $[\mathrm{diag}+aE_{h,h+1}^\theta]\in\mathcal{S}^\imath_{m|n,d}$ for $h\neq0,-1$. Therefore we have $m_A\in\mathcal{S}^\imath_{m|n,d}$ for any $A\in\Xi^\imath_{m|n,d}$. To summarize:
\begin{lem}
  For any $A\in\Xi^\imath_{m|n,d}$, we have $m_A\in\mathcal{S}^\imath_{m|n,d}$. Hence $\{m_A~|~A\in\Xi^\imath_{m|n,d}\}$ forms a monomial $\mathcal{A}$-basis for $\mathcal{S}^\imath_{m|n,d}$.
\end{lem}

Similar to \S~\ref{sec:bar},  we can define a bar involution on $\mathcal{S}^\imath_{m|n,d}$. It is easy to see that this bar involution on $\mathcal{S}^\imath_{m|n,d}$ can be identified with the restriction from the bar involution on $\mathcal{S}^\jmath_{m|n,d}$ via the inclusion $\mathcal{S}^\imath_{m|n,d}\subset\mathcal{S}^\jmath_{m|n,d}$.
Thus we have the following theorem about canonical basis for $\mathcal{S}^\imath_{m|n,d}$.
\begin{thm}\label{thm:i_canonical}
The canonical basis (relative to the standard basis $\{[A]~|~A\in\Xi^\imath_{m|n,d}\}$) for $\mathcal{S}^\imath_{m|n,d}$ is given by $\{\{A\}|A\in\Xi^\imath_{m|n,d}\}$.
\end{thm}

\subsection{Isomorphism theorem for $\mathbb{S}^\imath_{m|n,d}$}
Like Section~\ref{sec:Iso}, hereinafter we take $\mathbb{K}$ to be a field of characteristic $\neq2$ containing invertible elements $v$ and $q=v^2\neq0,1$. Denote $\mathbb{S}^\imath_{m|n,d}=\mathcal{S}^\imath_{m|n,d}\otimes_\mathcal{A}\mathbb{K}$.
Recall $\mathbb{V}_{m|n, <0}$ and $\mathbb{V}_{m|n, >0}$ from \eqref{def:space1} and \eqref{def:space2}, respectively.
We have the following canonical isomorphisms:
\begin{align*}
\mathbb{S}_{m|n,d}\simeq\mathrm{End}_{\mathbb{H}(\mathfrak{S}_d)}(\mathbb{V}^{\otimes d}_{m|n,> 0})\simeq\mathrm{End}_{\mathbb{H}(\mathfrak{S}_d)}(\mathbb{V}^{\otimes d}_{m|n,< 0}).
\end{align*}

Imitating the arguments in Section~\ref{sec:Iso}, we can get an isomorphism theorem for $\mathbb{S}^\imath_{m|n,d}$.
\begin{thm}\label{thm:i_iso}
If $f_d(q)$ is invertible in $\mathbb{K}$, then we have an isomorphism of $\mathbb{K}$-algebras:
\begin{align*}
\Phi: \mathbb{S}^\imath_{m|n,d}\rightarrow \bigoplus_{i=0}^d \mathbb{S}_{m|n,i}\otimes\mathbb{S}_{m|n,(d-i)}.
\end{align*}
\end{thm}

\subsection{Semisimplicity criteria for $\mathbb{S}^\imath_{m|n,d}$}
Similar to \S\ref{semcri}, the above isomorphism theorem helps us obtain semisimplicity criteria for $\mathbb{S}^\imath_{m|n,d}$ as follows.
\begin{thm}\label{thm:i_semi_B}
Suppose $f_d(q)$ is invertible in $\mathbb{K}$. The $\imath$Schur superalgebra $\mathbb{S}^\imath_{m|n,d}$ with $m,~n\geq 0$ is semisimple if and only if one of the following holds:
\begin{enumerate}
\item[(1)] $q$ is not a root of unity;
\item[(2)] $q$ is a primitive r-th root of unity with $r>d$;
\item[(3)] $(m,n)=(1,0)$.
\end{enumerate}
\end{thm}


\end{document}